\documentclass[12pt]{amsart}
\usepackage{amsmath, mathtools}
\usepackage{amsthm}
\usepackage{amssymb}
\usepackage{enumerate}
\usepackage{amscd}
\usepackage{color}
\usepackage{pb-diagram}
\usepackage{graphicx}
\usepackage[all, cmtip]{xy}
\usepackage{soul}%For the command \st{text} to cross out "text"
\usepackage[colorinlistoftodos]{todonotes}
\usepackage{esint}
\usepackage{hyperref}
\hypersetup{pdftex,colorlinks=true,allcolors=blue}
\usepackage{hypcap}
\usepackage[titletoc]{appendix}
\usepackage{comment}
\usepackage{float}
\restylefloat{table}
\usepackage{tikz}
\usetikzlibrary{calc,decorations.pathmorphing,shapes}

\usepackage[colorinlistoftodos]{todonotes}
\usepackage[normalem]{ulem}
\usepackage{soul,xcolor}
\setstcolor{red}
\makeatletter
\def\uwave{\bgroup \markoverwith{\lower3.5\p@\hbox{\sixly \textcolor{red}{\char58}}}\ULon}
\font\sixly=lasy6 % does not re-load if already loaded, so no memory problem.
\makeatother

\newenvironment{red}{\relax\color{red}}{\relax}
\newenvironment{blue}{\relax\color{blue}}{\hspace*{.5ex}\relax}

\newcommand{\ber}{\begin{red}}
	\newcommand{\er}{\end{red}}
\newcommand{\beb}{\begin{blue}}
	\newcommand{\eb}{\end{blue}}
\newcounter{sarrow}
\newcommand\xrsquigarrow[1]{%
	\stepcounter{sarrow}%
	\mathrel{\begin{tikzpicture}[baseline= {( $ (current bounding box.south) + (0,-0.5ex) $ )}]
			\node[inner sep=.5ex] (\thesarrow) {$\scriptstyle #1$};
			\path[draw,<-,decorate,
			decoration={zigzag,amplitude=0.7pt,segment length=1.2mm,pre=lineto,pre length=4pt}] 
			(\thesarrow.south east) -- (\thesarrow.south west);
	\end{tikzpicture}}%
}

\usepackage{amsrefs}

\theoremstyle{plain}
\newtheorem{lemma}{Lemma}%[section]
\newtheorem{prop}[lemma]{Proposition}
\newtheorem{coro}[lemma]{Corollary}
\newtheorem{theo}[lemma]{Theorem}
\newtheorem{rema}[lemma]{Remark}
\newtheorem{lemm}[lemma]{Lemma}
\newtheorem{mainthm}{Theorem}

\theoremstyle{definition}
\newtheorem{thm-Intro}{Theorem} 
\newtheorem{cor-Intro}{Corollary} 

\numberwithin{equation}{section}

\newcommand{\Ric}{\textup{Ric}}

\newcommand{\vol}{\textup{Vol}}

\newcommand{\p}{\partial}
\newcommand{\op}{\overline \partial}

\newcommand{\rsa}{\ \xrsquigarrow{(0,y,z)} \ }

\usepackage{fullpage}
\begin{document}
	\title[Equivalence of Invariant metrics via Bergman kernel]{Equivalence of Invariant metrics via Bergman kernel on complete noncompact  K\"ahler manifolds}
	
	\author{Gunhee Cho}
	\address{Department of Mathematics\\
		University of California, Santa Barbara\\
		552 University Rd, Isla Vista, CA 93117.}
	\email{gunhee.cho@math.ucsb.edu}

	\author{Kyu-Hwan Lee}
	\address{Department of Mathematics\\
		University of Connecticut\\
		%196 Auditorium Road,
		Storrs, CT 06269, U.S.A.}
	\email{khlee@math.uconn.edu}

	\begin{abstract} 
		We study equivalence of invariant metrics on noncompact K\"ahler manifolds with a complete Bergman metric of bounded curvature. Especially only the boundedness of the ratio between Bergman kernel and the $n$-times wedge product of Bergman metric in any fundamental domain of such a K\"ahler manifold is required to obtain the equivalence of the Bergman metric and the complete K\"ahler--Einstein metric. To demonstrate the effectiveness of this method, we consider a two-parameter family of $3$-dimensional bounded pseudoconvex domains \[ E_{p,\lambda}=\{(x,y,z)\in \mathbb{C}^3 ; (|x|^{2p}+|y|^2)^{1/{\lambda}}+|z|^2<1  \},\qquad p,\lambda>0.\] 
		For this family, boundary limits of the holomorphic sectional curvature of the Bergman metric are not well-defined, and hence previously known methods for comparison of invariant metrics do not work. Lastly, we provide an estimate of lower bound of the integrated Carath\'eodory--Reiffen metric on complete noncompact simply-connected K\"ahler manifolds with negative sectional curvature. 
	\end{abstract}
	
	\maketitle
	
	%\tableofcontents
	
	\section{Introduction}
	
	As the Bergman metric, the complete K\"ahler--Einstein metric of negative scalar curvature, the Kobayashi--Royden metric, and the Carath\'eodory--Reiffen metric are generalizations of the Poincar\'e--Bergman metric on the complex hyperbolic space, equivalence of these four invariant metrics on negatively curved complex manifolds has been studied in complex geometry.  In addition, since these four metrics have the property that any automorphism becomes an isometry \cite{HW93,WuYau20}, it makes sense to study them from the viewpoint of differential geometry. Hermitian metrics and Finsler metrics with this property are called {\em invariant metrics}. Some well-known classes having equivalence of these metrics are complex manifolds with uniform squeezing property, smoothly bounded strictly pseudoconvex domains in $\mathbb{C}^n$, and weakly pseudoconvex domains of finite type in $\mathbb{C}^2$ \cite{SKY09, CDW89}. In complex dimension $3$, the equivalence of these metrics breaks down for some weakly pseudoconvex domains with analytic boundary \cite{DKFJEHG84}. 
	
	In this context, D. Wu and S.\,T. Yau proved the following remarkable theorems based on the quasi-bounded geometry and Shi's estimate \cite{ShiWan-Xiong97} with K\"ahler--Ricci flow.  
	
	\begin{theo}[\cite{WuDaminYauShingTung20}, Corollary 7]\label{thm:Wu-Yau1} Let $(M,\omega)$ be a complete simply-connected noncompact K\"ahler manifold whose Riemannian sectional curvature is negatively pinched.  Then the base K\"ahler metric is uniformly equivalent to the Kobayashi--Royden metric, the Bergman metric and the complete K\"ahler--Einstein metric of negative scalar curvature.
	\end{theo}
	
	\begin{theo}[\cite{WuDaminYauShingTung20}, Theorems 2, 3]\label{thm:Wu-Yau2} Let $(M,\omega)$ be a complete  K\"ahler manifold whose holomorphic sectional curvature is negatively pinched.  Then the base K\"ahler metric is uniformly equivalent to the Kobayashi--Royden metric and the complete K\"ahler--Einstein metric of negative scalar curvature.
	\end{theo}
	
	As an interesting application of equivalence of invariant metrics, it is recently showed by the first-named author that the non-equivalence of invariant metrics can be used to show the non-existence of complete K\"ahler metric whose holomorphic sectional curvature is negatively pinched on pseudoconex domains in $\mathbb{C}^n$ under some conditions (see \cite{Cho23}).
	
	Based on Theorem~\ref{thm:Wu-Yau2}, one possible method to show the equivalence of the invariant metrics on a complete K\"ahler manifold $(M,\omega)$ is to prove that the holomorphic sectional curvature of $\omega$ has a negative range. As explicit formulas are recently obtained for the Bergman kernels on certain weakly pseudoconvex domains (e.g., see \cite{BT15, BT16-2, PJD13, JPDA94} and references therein), one could attempt to compute the holomorphic sectional curvature of the Bergman metric to establish the equivalence of the invariant metrics (for example, see \cite{GCYY20}). However, in general, it seems to be a daunting task to compute the holomorphic sectional curvature for nontrivial pseudoconvex domains even with explicit formulas of the Bergman kernels. 
	
	Indeed, for the bounded pseudoconvex domains, even for the  class of convex domains or strictly pseudoconvex domains, the curvature information of Bergman metric is known only near the boundary and not in the interior. The holomorphic sectional curvature of the Bergman metric has values between $-\infty$ and $+2$ \cite{KobayashiShoshichi59,DinewZywomir10}, but there is an example \cite{HerbortGregor07} of a semi-finite type pseudoconvex domain in which the holomorphic sectional curvature of Bergman metric blows up to $-\infty$. 
	
	Our main result in this paper is that, neither requiring the negative range of curvature as Wu--Yau theorems do, nor specifying the type of pseudoconvex domains, we provide a concrete approach to compare invariant metrics. Our method is based on knowledge of the Bergman kernel and can be applied to general bounded pseudoconvex domains $\Omega$ in $\mathbb{C}^n$ when an explicit description of the Bergman kernel near the boundary of $\Omega$ is available. 
	
	To state the main result (Theorem~\ref{thm:comparison}) below,	we define the fundamental domain $\widetilde{M}$ of a complex manifold $M$ to be the subset of $M$ which contains exactly one point from each of the orbits of the group action by the automorphism group of $M$. An automorphism $f$ of $M$ means $f$ and its inverse are holomorphic. 
	
	\begin{mainthm}\label{thm:comparison}
		Let $(M,\omega_B)$ be an $n$-dimensional noncompact K\"ahler manifold with a complete Bergman metric $\omega_B$ of bounded curvature, where $B$ denotes the Bergman kernel on $M$ (as the $(n,n)$-form). Then the following statements hold:
		
		1. Assume that $\frac{B}{\omega^n_{B}}$ is a bounded function for some fundamental domain $\widetilde{M}$. Here $\omega^n_{B} \coloneqq \omega_{B}\wedge \cdots \wedge \omega_{B}$ ($n$-times). Then there exist a complete K\"ahler--Einstein metric $\omega_{KE}$ of negative scalar curvature and a constant $C_1>0$ such that $\omega_{KE}$ is uniformly equivalent to $\omega_{B}$ by $C_1$, i.e.,  
		\begin{equation*}
			\frac{1}{C_1}\omega_{KE}(v,v)\leq \omega_{B}(v,v) \leq C_1\omega_{KE}(v,v) \qquad \text{ for all } v\in T'M.
		\end{equation*}
		
		2. Assume that there exists a compact subset $K$ in $M$ such that the holomorphic sectional curvature of $\omega_{B}$ is negative outside of $K$, and that $M$ is biholomorphically and properly embedded into $B_{N}$, $N\geq n$, where $B_{N}$ is the unit ball in $\mathbb{C}^N$. Then the Carath\'eodory--Reiffen metric $\gamma_{M}$ is not essentially zero, and the Bergman metric is uniformly equivalent to the Kobayashi--Royden metric, i.e., there exists $C_2>0$ such that
		\begin{equation*}
			\frac{1}{C_2}\chi_{M}(p;v) \leq \sqrt{\omega_{B}(v,v)} \leq C_2\chi_{M}(p;v) \qquad \text{ for all } v\in T_p'M, \  p\in M,
		\end{equation*}
		where $\chi_{M}$ is the Kobayashi--Royden metric on $M$. Moreover, if $N=n$, the Bergman metric is uniformly equivalent to the complete K\"ahler--Einstein metric of negative scalar curvature. 
	\end{mainthm}
	
	\begin{rema}\label{rmk:comparison-Kobayashi}
		
		Under the same assumptions of Theorem~\ref{thm:comparison}, but without additional assumptions of the first and second statements, we obtain the following from \cite{WuDaminYauShing-Tung2016}: there exists $C_0>0$, which only depends on $n$ and the curvature range of $\omega_{B}$, such that 
		\begin{equation*}
			\chi_{M}(p;v) \leq C_0 \sqrt{\omega_{B}(v,v)} \qquad \text{ for all } v\in T'_{p}M, \  p\in M.
		\end{equation*}
		(See Remark \ref{rem-details} for the details.)
	\end{rema}
	
	%	We prove the first statement of Theorem~\ref{thm:comparison} using Wu--Yau's quasi-bounded geometry and Shi's estimate on K\"ahler--Ricci flow and the second statement using the invariance of Bergman kernel and Bergman metric and an application of Chau's result \cite{ChauAlbert04}. 
	The second statement of Theorem~\ref{thm:comparison} differs from the Wu--Yau theorems (Theorems \ref{thm:Wu-Yau1} and \ref{thm:Wu-Yau2}) in that the Bergman metric's holomorphic sectional curvature is not required to be everywhere negative, but it still ensures the equivalence of invariant metrics. For the other assumption, we note that every bounded strictly pseudoconvex domain in $\mathbb{C}^n$ admits a proper holomorphic embedding into a ball (for example, see \cite[p.11]{ForstnerivcFranc18}). 
	
	To demonstrate the effectiveness of our method, we consider invariant metrics on a two-parameter family of $3$-dimensional bounded domains defined by
	\begin{equation}\label{eq:main-example}
		E_{p,\lambda}=\{(x,y,z)\in \mathbb{C}^3 ; (|x|^{2p}+|y|^2)^{1/{\lambda}}+|z|^2<1  \},\qquad p,\lambda>0.
	\end{equation}
	When $p=\lambda=1$, the domain $E_{p,\lambda}$ is the unit ball in $\mathbb{C}^3$. When $\lambda=1$ and $p\geq 1/2$, this reduces to the well-known convex egg (Thullen) domains whose invariant metrics are uniformly equivalent (\cite{GCYY20, MR3478940}). With other pairs of $(p,\lambda)$ for \eqref{eq:main-example}, the boundary limits of the holomorphic sectional curvature of the Bergman metric are not well-defined, so neither squeezing functions nor the Wu--Yau theorems can be applied.
	However, we show that Theorem \ref{thm:comparison} can be applied. For this purpose, we use  a concrete formula for the Bergman kernel of $E_{p,\lambda}$, which is obtained in \cite{BT15}. We also verify the Cheng's conjecture on $E_{p,\lambda}$ in the process of calculation. Namely, we show that the Bergman metric and the complete K\"ahler--Einstein metric is the same on $E_{p,\lambda}$  if and only if $p=\lambda=1$ (Proposition~\ref{prop:aysmp-Ric}).
	
	In the last section, we obtain a result on the Carath\'eodory--Reiffen metric which is missing in the Wu--Yau theorems. Classical invariant metrics include the Carath\'eodory--Reiffen metric whose definition is based on the existence of non-constant bounded holomorphic functions on noncompact complex manifolds. However, showing the existence of such functions still remains as a big challenge in hyperbolic complex geometry. 
	
	The upper bounds of the Carath\'eodory--Reiffen metric have been studied extensively. As for comparison between Carath\'eodory--Reiffen metric and the Bergman metric on the bounded domains, the first result is obtained by Qi-Keng Lu \cite{LookKH58} and then on manifolds by K. T. Hahn \cite{HahnKyongT76,HahnKyongT77}.  Further developments are made by T. Ahn, H. Gaussier and K. Kim \cite{AhnTaeyongGaussierHerveandKimKangTae16}. Very recently, a comparison of Carath\'eodory distance and K\"ahler--Einstein distance of Ricci curvature $-1$ for certain weakly pseudoconvex domains is established by the first-named author \cite{GC21}. 
	
	Our result in the last section is  a lower bound of the integrated Carath\'eodory--Reiffen metric (Theorem \ref{thm:lower-bounded-c-metric}). The positive lower bound of the Carath\'eodory--Reiffen metric is important in that it is the smallest invariant metric among invariant metrics \cite{GC21,JP13}, and it provides quantitative information about non-constant bounded holomorphic functions (also, see \cite{MCGCMGGY21}).
	
	\medskip
	
	The article is organized as follows: In Section 2, we review the definitions of the invariant metrics.  In the next section, we recall the quasi-bounded geometry and a result on comparison with the Kobayashi--Royden metric.  In Section 4, we apply Shi's estimate on K\"ahler--Ricci flow outside of a compact subset on noncompact K\"ahler manifold. In Section 5, we prove Theorem~\ref{thm:comparison} by generating a complete K\"ahler metric with negatively pinched holomorphic sectional curvature and applying the Wu--Yau theorems.  In Section 6, we perform explicit calculation on $E_{p,\lambda}$ for any $(p,\lambda)$ to verify the bounded curvature of the complete Bergman metric, and the hypothesis of Theorem~\ref{thm:comparison}-3. 
	In the last section, we prove Theorem \ref{thm:lower-bounded-c-metric} to obtain an integrated lower bound of the Carath\'eodory--Reiffen metric in the setting of Theorem~\ref{thm:Wu-Yau1}.

	\subsection*{Acknowledgments}
	GC is partially supported by Simons Travel funding. KHL is partially supported by a grant from the Simons Foundation (\#712100). Authors appreciate valuable comments from anonymous referees.

	\section{Preliminaries}
	Let $M$ be an $n$-dimensional complex manifold equipped with a complex structure $J$ and a Hermitian metric $g$. The complex structure $J : T_{\mathbb{R}}M \rightarrow T_{\mathbb{R}}M$ is a real linear endomorphism that satisfies for every $x \in M$, and $X,Y \in T_{\mathbb{R},x} M$, $g_x(J_x X,Y)=-g_x(X,J_xY)$, and for every $x \in M$, $J_x^2=-\mathbf{Id}_{T_x M} $. We decompose the complexified tangent bundle  $T_{\mathbb{R}}M\otimes_{\mathbb{R}}\mathbb{C}=T'M\oplus \overline{T'M}$, where $T'M$ is the eigenspace of $J$ with respect to the eigenvalue $\sqrt{-1}$ and $\overline{T'M}$ is the eigenspace of $J$ with respect to the eigenvalue $-\sqrt{-1}$. We can regard $v, w$ as real tangent vectors, and $\eta,\xi$ as corresponding holomorphic $(1,0)$ tangent vectors under the $\mathbb{R}$-linear isomorphism $T_{\mathbb{R}}M \rightarrow T'M$, i.e. $\eta=\frac{1}{{\sqrt{2}}}(v-\sqrt{-1}Jv), \xi=\frac{1}{{{\sqrt{2}}}}(w-\sqrt{-1}Jw)$.
	
	A {Hermitian metric} on $M$ is a positive definite Hermitian inner product
	\[
	g_p : T'_p M \otimes \overline{T'_p M} \rightarrow \mathbb{C}
	\]
	which varies smoothly for each $p \in M$. The metric $g$ can be decomposed into the real part denoted by $\operatorname{Re}(g)$, and the imaginary part denoted by $\operatorname{Im}(g)$. The real part $\operatorname{Re}(g)$ induces an inner product called the induced Riemannian metric of $g$, an alternating $\mathbb{R}$-differential $2$-form. Define the $(1,1)$-form $\omega \coloneqq -\frac{1}{2}\operatorname{Im}(g)$, which is called the fundamental $(1,1)$-form of $g$ or the K\"ahler metric. In local coordinates this form can written as
	\[
	\omega=\frac{\sqrt{-1}}{2}\sum_{i,j=1}^{n}g_{i\overline{j}}dz_i \wedge d\overline{z_j}.
	\]
	
	The components of the curvature $4$-tensor of the Chern connection associated with the Hermitian metric $g$ are given by
	\begin{align*}
		& R_{i\overline{j}k\overline{l}}:= R(\frac{\partial}{\partial z_i},\frac{\partial}{\partial z_i},\frac{\partial}{\partial z_i},\frac{\partial}{\partial z_i})
		\\
		& =g \left( \nabla^c_{\frac{\partial}{\partial z_i} } \nabla^c_{\frac{\partial}{\partial \overline{z_j}} }\frac{\partial}{\partial z_k} -\nabla^c_{\frac{\partial}{\partial \overline{z_j}} } \nabla^c_{\frac{\partial}{\partial z_i}}\frac{\partial}{\partial z_k} -\nabla^c_{[{\frac{\partial}{\partial z_i} },{\frac{\partial}{\partial \overline{z_j}} }]} {\frac{\partial}{\partial z_k} }, {\frac{\partial}{\partial \overline{z_l}}}\right)
		\\
		& =-\frac{\partial^2 g_{i\overline{j}}}{\partial z_k \partial \overline{z}_l}+\sum_{p,q=1}^{n}g^{q\overline{p}}\frac{\partial g_{i\overline{p}}}{\partial z_k}\frac{\partial g_{q\overline{j}}}{\partial \overline{z}_l},
	\end{align*}
	where $i,j,k,l\in \left\{1,\dots, n \right\}$.
	
	The holomorphic sectional curvature with the unit direction $\eta$ at $x \in M$ (i.e., $g_{\omega}(\eta,\eta)=1$) is defined by
	\begin{equation*} H(g)(x,\eta)=R(\eta,\overline{\eta},\eta,\overline{\eta})=R(v,Jv,Jv,v),
	\end{equation*}
	where $v$ is the real tangent vector corresponding to $\eta$. We will often write $H(g)(x,\eta)=H(g)(\eta)=H(\eta)$.
	The Ricci tensor of a K\"ahler metric $\omega$ is defined by \begin{equation*} \Ric(\omega):=-\sqrt{-1}\partial \overline{\partial} \log \det (g).
	\end{equation*} 
	Given any complex manifold $M$, for each $p \in M$ and a tangent vector $v$ at $p$, define the Carath\'eodory--Reiffen metric and the Kobayashi--Royden metric by
	\[\gamma_{M}(p;v):=\sup\left\{|df(p)(v)|; \ f : M \rightarrow \mathbb{D},  f(p)=0, \text{$f$ holomorphic} \right\}, \]
	\[ \chi_{M}(p;v):=\inf\left\{\frac{1}{R}; \ f :  R\mathbb{D} \rightarrow M, f(0)=p, df(\tfrac{\partial}{\partial z}{|_{z=0}})=v, \text{$f$ holomorphic} \right\}, \]
	respectively.
	
	The Bergman metric is defined in terms of the Bergman kernel. Let $\Lambda^{(n,0)}M$ be the space of smooth complex differential $(n,0)$-forms on $M$. For $\varphi,\psi\in \Lambda^{(n,0)}M$, define \[ \langle \varphi,\psi \rangle = (-1)^{n^2/2}\int_{M} \varphi \wedge \overline{\psi}, \] and \[ ||\varphi||=\sqrt{\langle \varphi,\varphi \rangle}.\] Let $L^2_{(n,0)}$ be the completion of \[\left\{\varphi \in \Lambda^{(n,0)}M; ||\varphi||<+\infty  \right\} \] with respect to $||\cdot||$. Then $L^2_{(n,0)}$ is a separable Hilbert space with respect to the inner product $\langle\cdot,\cdot \rangle$. 
	
	Define $\mathcal{H}=\left\{\varphi \in L^2_{(n,0)} ; \varphi \text{ is holomorphic}  \right\}$. Suppose $\mathcal{H} \neq 0$. Let $\left\{e_j \right\}_{j\geq 0}$ be an orthonormal basis of $\mathcal{H}$ with respect to $\langle\cdot,\cdot \rangle$. Then the $2n$-form on $M\times M$, defined by \[{B}(x,y):=\sum_{j\geq 0}e_j(x)\wedge \overline{e}_j(y), 
	\qquad x,y \in M, \] is called the {{Bergman kernel}} of $M$. Suppose for some point $p \in M$, we have ${B}(p,p)\neq 0$. Write ${B}(z,z)=b(z,z) dz_1 \wedge \cdots \wedge dz_n \wedge d\overline{z}_1\wedge \cdots \wedge d\overline{z}_n$ in terms of local coordinates $(z_1,\cdots,z_n)$. Define \[\omega_{B}(z):=\sqrt{-1}\partial \overline{\partial}\log b(z,z). \] If the real $(1,1)$-form $\omega_{B}$ is positive definite, we call the corresponding Hermitian metric $g^B_{M}$ the {{Bergman metric}}. By definition, $g^B_{M}$ is K\"ahler.
	
	Lastly, the {K\"ahler--Einstein metric} $\omega_{KE}$ means the K\"ahler metric which is also the Einstein metric, and the {K\"ahler--Einstein metric} of the negative scalar curvature becomes an invariant metric.  
	
	We will use the following lemma to prove Theorem~\ref{thm:comparison}:
	
	\begin{lemm}[{\cite[Lemma 19]{WuDaminYauShingTung20}}] \label{Kobayashi-upper-bound}
		Let $(M,\omega)$ be a Hermitian manifold such that the holomorphic sectional curvature has the upper bound $-\kappa<0$. Then the Kobayashi--Royden metric satisfies
		\begin{equation*}
			\chi_{M}(x,v)\geq \sqrt{\frac{\kappa}{2}}|v|_{\omega},
		\end{equation*}
		for each $x\in M,v\in T'_x M$. 
	\end{lemm}
	
	\section{Quasi-bounded geometry}
	In this section, we review some results from Section 2 in \cite{WuDaminYauShingTung20}.  
	
	\medskip
	
	The notion of quasi-bounded geometry is introduced by S.\,T. Yau and S.\,Y. Cheng (\cite{CY80}). Let $(M,\omega)$ be an $n$-dimensional complete K\"ahler manifold. For a point $p\in M$, let $B_{\omega}(p;\rho)$ be the open geodesic ball centered at $p$ in $M$ of radius $\rho$; we omit the subscript $\omega$ if there is no peril of confusion. Denote by $B_{\mathbb{C}^n}(r)$ the open ball centered at the origin in $\mathbb{C}^n$ of radius $r$ with respect to the standard metric $\omega_{\mathbb{C}^n}$. 
	
	\begin{def}\label{def_quasi_bounded}
		An $n$-dimensional K\"ahler manifold $(M,\omega)$ is said to have {\em quasi-bounded geometry} if there exist two constants $r_2>r_1>0$ such that for each point $p\in M$, there is a domain $U\subset \mathbb{C}^n$ and a nonsingular holomorphic map $\psi : U \rightarrow M$ satisfying 
		
		(1) $B_{\mathbb{C}^n}(r_1)\subset U \subset B_{\mathbb{C}^n}(r_2)$ and $\psi(0)=p$;
		
		(2) there exists a constant $C>0$ depending only on $r_1,r_2,n$ such that 
		\begin{equation}\label{eq:2.1}
			C^{-1}\omega_{\mathbb{C}^n} \leq \psi^{*}(\omega) \leq C\omega_{\mathbb{C}^n} \quad \text{ on } U;
		\end{equation}
		
		(3)	for each integer $l\geq 0$, there exists a constant $A_l$ depending only on $l,n,r_1,r_2$ such that 
		\begin{equation}\label{eq:2.2}
			\sup_{x\in U}\left |\frac{\partial^{|\nu|+|\mu|}g_{i\overline{j}} }{\partial v^{\mu}\, \partial \overline{v}^{\nu} } \right |\leq A_l, \text{ for all } |\mu|+|\nu|\leq l, 
		\end{equation}
		where $g_{i\overline{j}}$ are the components of $\psi^{*}\omega$ on $U$ in terms of the natural coordinates $(v^1,\cdots,v^n)$, and $\mu,\nu$ are multiple indices with $|\mu|=\mu_1+\cdots +\mu_n$.
		We call $r_1$ a {\em radius} of quasi-bounded geometry. 
	\end{def}

	By applying the $L^2$-estimate, the following theorem is proved. 
	
	\begin{theo}[\cite{WuDaminYauShingTung20}, Theorem 9]\label{thm:equiv-quasi-bounded-geometry} Let $(M,\omega)$ be a complete K\"ahler manifold.  Then the manifold $(M,\omega)$ has quasi-bounded geometry if and only if for each integer $q\geq 0$, there exists a constant $C_q>0$ such that 
		\begin{equation}\label{eq:bg}
			\sup_{p \in M}|\nabla^q R_m|\leq C_q,
		\end{equation}
		where $R_m=\{R_{i\overline{j}k\overline{l} } \}$ denotes the curvature tensor of $\omega$. In this case, the radius of quasi-bounded geometry depends only on $C_0$ and the dimension of $M$. 
	\end{theo}
	
	Also, we will use the following lemma: 
	
	\begin{lemm}[{\cite[Lemma 20]{WuDaminYauShingTung20}}] \label{quasi-bounded-Kobayashi}
		Suppose a complete K\"ahler manifold $(M,\omega)$ has quasi-bounded geometry. Then the Kobayashi--Royden metric satisfies
		\begin{equation*}
			\chi_{M}(x,v)\leq C|v|_{\omega},
		\end{equation*}
		for each $x\in M,v\in T'_x M$, where $C$ depends only on the radius of quasi-bounded geometry of $(M,\omega)$. 
	\end{lemm}
	
	\section{The Maximum Principle and Shi's estimate on K\"ahler--Ricci flow}
	
	Let $(M,\widetilde{\omega})$ be an $n$-dimensional complete noncompact K\"ahler manifold. Suppose for some constant $T>0$ there is a smooth solution $\omega(x,t)>0$ for the evolution equation
	\begin{equation}\label{A.1}
		\begin{cases} \frac{\partial}{\partial t}g_{\alpha \overline{\beta}}(x,t)=-4R_{\alpha \overline{\beta}}(x,t) & \text{ on } M\times [0,T], \\
			g_{\alpha \overline{\beta}}(x,0)=\widetilde{g}_{\alpha \overline{\beta}}(x) & \text{$x \in M$,} \end{cases}
	\end{equation}
	where $g_{\alpha \overline{\beta}}(x,t)$ and $\widetilde{g}_{\alpha \overline{\beta}}$ are the metric components of $\omega(x,t)$ and $\widetilde{\omega}$, respectively. Assume that the curvature $R_m(x,t)=\left\{R_{\alpha \overline{\beta}\gamma \overline{\delta}(x,t) } \right\}$ of $\omega(x,t)$ satisfies
	\begin{equation}\label{A.2}
		\sup_{M\times [0,T]}|R_m(x,t)|^2\leq k_0
	\end{equation}
	for some constant $k_0>0$. 
	
	The following lemma is an extension of Lemma 15 in \cite{WuDaminYauShingTung20} to the case of complement of compact subset. Though the proof is similar, we provide some details to indicate where modifications are needed for the complement.
	
	\begin{lemma}\label{Max-Principle} With the above assumptions, suppose a smooth tensor $\left\{W_{\alpha \overline{\beta}\gamma \overline{\delta}(x,t) } \right\}$ on $M$ with complex conjugation $W_{\alpha \overline{\beta}\gamma \overline{\delta}(x,t) }=W_{\beta \overline{\alpha}\delta \overline{\gamma}(x,t) }$ satisfies 
		\begin{equation}\label{A.3}
			\left(\tfrac{\partial}{\partial t}W_{\alpha \overline{\beta}\gamma \overline{\delta}(x,t) }\right)\eta^{\alpha}\overline{\eta}^{\beta}\eta^{\gamma}\overline{\eta}^{\delta}\leq (\triangle W_{\alpha \overline{\beta}\gamma \overline{\delta}})\eta^{\alpha}\overline{\eta}^{\beta}\eta^{\gamma}\overline{\eta}^{\delta}+C_1|\eta|^4_{\omega(x,t)},
		\end{equation}
		for all $x\in M, \eta\in T'_x M, 0\leq t \leq T$, where $\triangle\equiv 2g^{\alpha \overline{\beta}}(x,t)(\nabla_{\overline{\beta}}\nabla_{\alpha}+\nabla_{\alpha}\nabla_{\overline{\beta}})$ and $C_1$ is a constant. Let  
		\begin{equation*}
			h(x,t)=\max \left\{ W_{\alpha \overline{\beta}\gamma \overline{\delta}}\eta^{\alpha}\overline{\eta}^{\beta}\eta^{\gamma}\overline{\eta}^{\delta}; \eta\in T'_xM, |\eta|_{\omega(x,t)}=1  \right\},
		\end{equation*}
		for all $x\in M$ and $0\leq t\leq T$. For any compact subset $K$ in $M$, suppose 
		\begin{equation}
			\sup_{x\in M, 0\leq t \leq T}|h(x,t)|\leq C_0, \label{A.4}
		\end{equation}
		\begin{equation}
			\sup_{M\backslash K} h(x,0)\leq -\kappa,  \label{A.5}
		\end{equation}
		for some constants $C_0>0$ and $\kappa$. Then, 
		\begin{equation*}
			h(x,t)\leq (8C_0\sqrt{nk_0}+C_1)t-\kappa,
		\end{equation*}
		for all $x\in M\backslash K$ and $0\leq t \leq T$. 	
	\end{lemma}
	
	\begin{proof}
		Denote 
		\begin{equation}
			C=8C_0\sqrt{nk_0}+C_1>0. \label{A.6}
		\end{equation}
		Suppose 
		\begin{equation}
			h(x_1,t_1)-Ct_1+\kappa>0, \label{A.7}
		\end{equation}
		for some $(x_1,t_1)\in M\backslash K \times [0,T]$. Then by \eqref{A.4} we have $t_1>0$. Under the conditions \eqref{A.1} and \eqref{A.2}, it follows from \cite{ShiWan-Xiong97} that there exists a function $\theta$ such that 
		\begin{equation}
			0<\theta(x,t)\leq 1, \text{ on } M\times [0,T],\label{A.8}
		\end{equation}
		\begin{equation}
			\frac{\partial \theta}{\partial t}-\triangle_{\omega(x,t)}\theta +2\theta^{-1}|\nabla \theta|^2_{\omega(x,t)}\leq -\theta \text{ on } M\times [0,T], \label{A.9}
		\end{equation}
		\begin{equation}
			\frac{C^{-1}_2}{1+d_0(x_0,x)}\leq \theta(x,t)\leq \frac{C_2}{1+d_0(x_0,x)} \text{ on } M\times [0,T],\label{A.10}
		\end{equation}
		where $x_0$ is a fixed point in $M$, $d_0(x,y)$ is the geodesic distance between $x$ and $y$ with respect to $\omega(x,0)$, and $C_2>0$ is a constant depending only on $n,k_0$ and $T$. 
		
		Let 
		\begin{equation*}
			m_0=\sup_{M\backslash K,0\leq t \leq T}\left([h(x,t)-Ct+\kappa]\theta(x,t)\right).
		\end{equation*}
		Then $0<m_0\leq C_0+|\kappa|$ by \eqref{A.4},\eqref{A.7}, and \eqref{A.8}. Denote  
		\begin{equation*}
			\Lambda = \frac{2C_2(C_0+CT+|\kappa|) }{m_0}>0. 
		\end{equation*} 
		Then, for any $x\in M\backslash K$ with $d_0(x_0,x)\geq \Lambda$,
		we have \begin{equation*}
			|(h(x,t)-Ct+\kappa)\theta(x,t)|\leq \frac{C_2(C_0+CT+|\kappa|)}{1+d_0(x,x_0)}\leq \frac{m_0}{2}. 
		\end{equation*}
		
		It follows that the function $(h-Ct+\kappa)\theta$ must attain its supremum $m_0$ on the compact set $\overline{B(x_0;\Lambda)}\times [0,T] \subset M\backslash K \times [0,T]$, where $\overline{B(x_0;r)}$ denotes the closure of the geodesic ball with respect to $\omega(x,0)$ centered at $x_0$ of radius $r$. Let 
		\begin{equation*}
			f(x,\eta,t)=\frac{W_{\alpha \overline{\beta}\gamma \overline{\delta}\eta^{\alpha}\overline{\eta}^{\beta}\eta^{\gamma}\overline{\eta}^{\delta}}}{|\eta|^4_{\omega(x,t)}} -Ct+\kappa,
		\end{equation*}
		for all $(x,t)\in M\backslash K \times [0,T], \eta \in T'_x M \backslash \left\{ 0 \right\}$. Then there exist $x_{*}, \eta_{*}, t_{*}$ with $x_{*}\in \overline{B(x_0;r)}, 0\leq t_{*}\leq T, \eta_{*}\in T'_{x_{*}}M$ and $|\eta_{*}|_{\omega(x_{*},t_{*})}=1$, such that 
		\begin{equation*}
			m_0=f(x_{*}, \eta_{*}, t_{*})\theta(x_{*}, t_{*})=\max_{\mathcal{S}_t\times [0,T]}(f\theta ),
		\end{equation*}
		where $\mathcal{S}_t=\left\{(x,\eta)\in T'M; x\in M, \eta\in T'_x M, |\eta|_{\omega(x,t)}=1  \right\}$. Since $h(.,0)$ is a continuous function on $M$, either $x_{*}\in M\backslash K$ or $x_{*}\in \partial K$, $t_{*}>0$ by \eqref{A.5}. Now we extend $\eta_{*}$ to a smooth vector field using the same argument as in the proof of Lemma 15 in \cite{WuDaminYauShingTung20}. Since $f\theta=f(x,\eta(x),t)\theta(x,t)$ attains its maximum at $(x_{*},t_{*})$, we have
		\begin{equation}\label{A.12}
			\tfrac{\partial }{\partial t}(f\theta)\geq 0, \text{ } \nabla (f \theta)=0, \text{ } \triangle (f\theta)\leq 0 \quad \text{ at } (x_{*},t_{*}).
		\end{equation}
		From \eqref{A.12} and \eqref{A.9}, one can see that at the point $(x_{*}, t_{*})$,
		we have \begin{equation*}
			0\leq \tfrac{\partial }{\partial t} (f\theta)=-m_0<0
		\end{equation*}
		(for details, see \cite{WuDaminYauShingTung20}). This yields a contradiction and the proof is completed.
	\end{proof}
	
	The following lemma is an extension of Lemma 13 in \cite{WuDaminYauShingTung20} to the case of complement of a compact subset.
	
	\begin{lemm}\label{lem:Wang-Xiong Shi}
		Let $(M,\omega)$ be an $n$-dimensional complete noncompact K\"ahler manifold. Let $K$ be a compact set in $M$ such that
		\begin{equation}\label{3.1}
			-\kappa_2\leq H(\omega) \leq -\kappa_1 <0 \text{ on } M\backslash K,
		\end{equation} 
		where $H(\omega)$ is the holomorphic sectional curvature and $\kappa_1,\kappa_2$ are positive constants. Then there exists another K\"ahler metric $\widetilde{\omega}$ such that
		\begin{equation}\label{3.2}
			C^{-1}\omega \leq \widetilde{\omega} \leq C \omega \qquad \qquad \text{ on } M,
		\end{equation}	 
		\begin{equation}\label{3.3}
			-\widetilde{\kappa_2}\leq H(\widetilde{\omega}) \leq -\widetilde{\kappa_1} <0 \qquad \text{ on } M\backslash K,
		\end{equation}
		\begin{equation}\label{3.4}
			\sup_{p \in M}|\widetilde{\nabla}^q \widetilde{R_m}|\leq C_q \qquad \qquad \text{ on } M,
		\end{equation}
		where $\widetilde{\nabla}^q$ denotes the $q$-th order covariant derivative of $\widetilde{R_m}$ with respect to $\widetilde{\omega}$, and the positive constants $C=C(n)$, $\widetilde{\kappa_j}=\widetilde{\kappa_j}(n,\kappa_1,\kappa_2)$, $j = 1,2$, $C_q = C_q(n, q, \kappa_1, \kappa_2)$ depend only on the parameters in their parentheses.
	\end{lemm}
	The conditions \eqref{3.2} and \eqref{3.4} appear in \cite{ShiWan-Xiong97,WuDaminYauShingTung20}. We provide below details for the pinching estimate. 
	\begin{proof}
		From the short time existence of the K\"ahler--Ricci flow \cite{ShiWan-Xiong97}, the equation \eqref{A.1} admits a smooth solution $\left\{g_{\alpha \overline{\beta}}(x,t) \right\} $ for all $0\leq t \leq T$. The curvature $R_m(x,t)$ satisfies 
		\begin{equation}\label{3.5}
			\sup_{x\in M}|\nabla^{q}R_m(x,t)|^2\leq \frac{C(q,n,K)(\kappa_2-\kappa_1)^2}{t^q}, \qquad 0<t\leq \frac{\theta_0(n,K)}{\kappa_2-\kappa_1}\equiv T,
		\end{equation} for each nonnegative integer $q$,
		where $C(q,n,k)>0$ is a constant depending only on $q$, $K$ and $n$, and $\theta_0(n,K)>0$ is a constant depending only on $n$ and $K$.  
		
		From the evolution equation of the curvature tensor (see  \cite{ShiWan-Xiong97,WuDaminYauShingTung20}), we have
		\begin{align*}
			\frac{\partial}{\partial t}R_{\alpha \overline{\beta}\gamma \overline{\delta}}&=4\triangle R_{\alpha \overline{\beta}\gamma \overline{\delta}}+4g^{\mu \overline{\nu}}g^{\rho \overline{\tau}}(R_{\alpha \overline{\beta}\mu \overline{\tau}}R_{\gamma \overline{\delta}\rho \overline{\nu}}+R_{\alpha \overline{\delta}\mu \overline{\tau}}R_{\gamma \overline{\beta}\rho \overline{\nu}}-R_{\alpha \overline{\nu}\gamma \overline{\tau}}R_{\mu \overline{\beta}\rho \overline{\delta}})\\
			&-2g^{\mu \overline{\nu}}(R_{\alpha \overline{\nu}}R_{\mu \overline{\beta}\rho \overline{\tau}}+R_{\mu \overline{\beta}}R_{\alpha \overline{\nu}\rho \overline{\tau}}+R_{\gamma \overline{\nu}}R_{\alpha \overline{\beta}\mu \overline{\tau}}+R_{\mu \overline{\delta}}R_{\alpha \overline{\beta}\rho \overline{\nu}}),
		\end{align*}
		where $\triangle \equiv \triangle_{\omega(x,t)}=\frac{1}{2}g^{\alpha \overline{\beta}}(x,t)(\nabla_{\overline{\beta}}\nabla_{\alpha}+\nabla_{\alpha}\nabla_{\overline{\beta}})$. It follows that 
		\begin{align}\label{3.7}
			&(\frac{\partial}{\partial t}R_{\alpha \overline{\beta}\gamma \overline{\delta}})\eta^{\alpha}\overline{\eta}^{\beta}\eta^{\gamma}\overline{\eta}^{\delta}\\
			&\leq 4(\triangle R_{\alpha \overline{\beta}\gamma \overline{\delta}})\eta^{\alpha}\overline{\eta}^{\beta}\eta^{\gamma}\overline{\eta}^{\delta}+C_1(n)|\eta|^4_{g_{\alpha \overline{\beta}}}(x,t)|R_m(x,t)|^2_{\omega(x,t)}\\ \nonumber
			&\leq 4(\triangle R_{\alpha \overline{\beta}\gamma \overline{\delta}})\eta^{\alpha}\overline{\eta}^{\beta}\eta^{\gamma}\overline{\eta}^{\delta}+\widetilde{C_1}(n,K)(\kappa_2-\kappa_1)^2|\eta|^4_{\omega(x,t)}, \nonumber
		\end{align}
		by \eqref{3.5} with $q=0$. Let 
		\begin{equation*}
			H(x,\eta,t)=\frac{R_{\alpha \overline{\beta}\gamma \overline{\delta}})\eta^{\alpha}\overline{\eta}^{\beta}\eta^{\gamma}\overline{\eta}^{\delta}}{|\eta|^4_{\omega(x,t)}}.
		\end{equation*}
		Then by \eqref{3.1} and \eqref{3.5}, 
		\begin{equation*}
			H(\widetilde{\omega}) \leq -\widetilde{\kappa_1} <0 \text{ on } M\backslash K,
		\end{equation*}
		\begin{equation*}
			|H(x,\eta,t)|\leq |R_m(x,t)|_{\omega(x,t)}\leq C_0(n,K)(\kappa_2-\kappa_1).
		\end{equation*}
		To apply the maximum principle, let us denote
		\begin{equation*}
			h(x,t)=\max\left\{H(x,\eta,t); |\eta|_{\omega(x,t)=1} \right\},
		\end{equation*}
		for all $x\in M$ and $0\leq t \leq \frac{\theta(n,K)}{\kappa_2-\kappa_1}$. Then $h$ with \eqref{3.7} satisfies the three conditions in Lemma~\ref{Max-Principle}. Then 
		\begin{equation*}
			H(x,\eta,t)\leq h(x,t)\leq -\frac{\kappa_1}{2}<0,
		\end{equation*}
		for all $0<t\leq t_0:=\min\left\{\frac{\kappa_1}{2\widetilde{C_1}(n,K)(\kappa_2-\kappa_1)^2},\frac{\theta_0(n,K)}{\kappa_2-\kappa_1} \right\}$. 
		Since the curvature tensor is bounded by \eqref{3.5} with $q=0$, the complete K\"ahler metric $\omega(x,t)=\frac{\sqrt{-1}}{2}g_{\alpha \overline{\beta}}(x,t) dz^{\alpha}\wedge d\overline{z}^{\beta}$ is a desired metric for an arbitrary $t\in (0,t_0]$. 
	\end{proof}

	\section{Generation of K\"ahler metrics  with negative \\ holomorphic sectional curvature}
	
	In this section, after establishing a proposition below, we prove Theorem \ref{thm:comparison}.
	
	\begin{prop}\label{prop:negative_curvature}
		Given an $n$-dimensional K\"ahler manifold $(M,\omega)$, assume that there exists a compact subset $K$ in $M$ such that the holomorphic sectional curvature of $\omega$ is negative outside of $K$, and $M$ is biholomorphically  and properly embedded into $B_{N}, N\geq n$, where $B_{N}$ is the unit ball in $\mathbb{C}^N$. Then there exists a complete K\"ahler metric $\widetilde{\omega}$ whose holomorphic sectional curvature has a negative upper bound and $\widetilde{\omega}\geq \omega$. 
	\end{prop}
	\begin{proof}
		
		From the holomorphic embedding $M \hookrightarrow B_{N}$, consider a K\"ahler metric of the form \[\omega_m:=m \omega_{P}+\omega, \quad m>0,\] where $\omega_{P}$ is the Poincar\'e metric of the unit ball $B_{N}$ in $\mathbb{C}^N$. It is clear that $\omega_m\geq \omega$ for each $m>0$. From the decreasing property of the holomorphic sectional curvature, $\omega_{P}$ restricted to $M$ has a negative holomorphic sectional curvature \cite{WU73}. From Lemma 4 of \cite{WU73}, we may assume that the holomorphic sectional curvature of $\omega_m$ is the Gaussian curvature on some embedded Riemann surfaces in $M$. Recall that for a Hermitian metric $G$ on a Riemann surface, the holomorphic sectional curvature of $G$ is the Gaussian curvature $H(g)=-\frac{1}{g}\frac{\partial^2 \log g}{\partial z \overline{\partial}z}$ of $G$ for some positive smooth function $g=g(z,\overline{z})$. In this case, the holomorphic sectional curvature $H(G,t)$ becomes a real-valued function independent of the unit vector $t$. Thus we write $H(G)$ instead of $H(G,t)$. 
		
		From \cite[Proposition 3.1]{KobayashiShoshichi05}, for any positive functions $f$ and $g$ with $m>0$, 
		\begin{align*}
			H(f+mg)&\leq \frac{f^2}{(f+mg)^2}H(f)+\frac{m^2 g^2}{(f+mg)^2}H(mg)\\
			&=\frac{f^2}{(f+mg)^2}H(f)+\frac{m g^2}{(f+mg)^2}H(g).
		\end{align*}
		From here, we can deduce that $H(\omega_m)$ becomes negative on $K$ by taking sufficiently large $m$. Since $H(\omega_m)$ is negative on $M \backslash K$, we are done. 
	\end{proof}
	
	\begin{proof}[Proof of Theorem~\ref{thm:comparison}]
		
		For the first statement, we fix a fundamental domain $\widetilde{M}$ and define a function $f : M \rightarrow \mathbb{C}$ by $f(z):=\frac{B(z)}{\omega^n_{B}(z)}$. Since the numerator and the denominator are smooth $(n,n)$-forms, the function $f$ is well-defined and clearly smooth. Note that the Bergman kernel and the Bergman metric are invariant under the automorphism group of $M$. Thus the boundedness assumption of $f$ on $\widetilde{M}$ implies the boundedness of $f$ on $M$, and we have a function $f$ which is smooth and bounded on $M$ satisfying $\Ric_{i\overline{j}} + g_{i\overline{j}}=f_{i\overline{j}}$ for each $i,j$, where we denote the Bergman metric in local coordinates by $(g_{i\overline{j}})$. Now we apply the main theorem in \cite{ChauAlbert04}, and the conclusion follows. 
		
		The first part of the second statement follows from Lemma~\ref{Kobayashi-upper-bound}, Lemma~\ref{quasi-bounded-Kobayashi} and Proposition~\ref{prop:negative_curvature} with the fact that for each $m>0$, 
		\begin{equation*}
			\omega_{B}\leq \widetilde{\omega},
		\end{equation*}
		where $\widetilde{\omega}$ is defined in Proposition~\ref{prop:negative_curvature}. 
		For the second part of the case $N=n$, the metric $ \widetilde{\omega}$ has the bounded curvature. Then one can solve the complex Monge--Ampere equation by following Wu--Yau's approach (see Lemma 31 and Theorem 3 in \cite{WuDaminYauShingTung20}). 
	\end{proof}
	
	\begin{rema}
		When $N>n$, the holomorphic sectional curvature $\widetilde{\omega}$ does not need to be bounded below because of the presence of the second fundamental form (see \cite{WU73}). 
	\end{rema}
	
	\begin{rema} \label{rem-details}
		If ~\eqref{3.1} is replaced by 
		\begin{equation*}
			-\kappa_2\leq H(\omega) \leq -\kappa_1  \text{ on } M \quad \text{ for } \kappa_1\in \mathbb{R},
		\end{equation*}
		then ~\eqref{3.2} and ~\eqref{3.4} still follow from the original Shi's argument. Combining it with Lemma~\ref{quasi-bounded-Kobayashi} and Lemma~\ref{lem:Wang-Xiong Shi}, we obtain a proof of the statement in Remark \ref{rmk:comparison-Kobayashi}. Indeed, by applying Shi's estimate on K\"ahler--Ricci flow with the short-time existence, we can generate a complete K\"ahler metric $\omega$ such that any order of covariant derivatives of the curvature tensor is bounded, and $\omega$ is equivalent to the Bergman metric $\omega_{B}$. Then by the characterization of quasi-bounded geometry of Wu--Yau \cite{WuDaminYauShingTung20}, $\omega$ admits a quasi-bounded geometry, and the statement in Remark~\ref{rmk:comparison-Kobayashi} follows from Lemma~\ref{quasi-bounded-Kobayashi}. 
	\end{rema}
	
	\section{Domain $E_{p,\lambda}$}
	
	In this section, we consider the domain \[ E_{p,\lambda}=\{(x,y,z)\in \mathbb{C}^3 ; (|x|^{2p}+|y|^2)^{1/{\lambda}}+|z|^2<1  \},\qquad p,\lambda>0,\] and perform necessary computations to examine the comparison of invariant metrics through verification of the hypotheses in Theorem~\ref{thm:comparison}. 
	
	First, we take a suitable compact set $K \subset E_{p,\lambda}\cup \partial E_{p,\lambda}$ that satisfies the conditions in Theorem \ref{thm:comparison}. Since any point $(x,y,z)\in \mathbb{C}^3$ can be realized as
	\begin{equation*}
		|x|<r(z,y)=\left((1-|z|^2)^{\lambda}-|y|^2 \right)^{\frac{1}{2p}},
	\end{equation*} 
	with a fixed pair $(y,z)$, the point $(x,y,z)$ can be mapped biholomorphically onto the form $(0,y,z)$ through the automorphism of one-dimensional disc with the radius $r(y,z)$ centered at the origin. Then using rotations, we can make the other two entries to have non-negative real-values. Since all these transformations are automorphisms of $E_{p,\lambda}$, we take the compact set: 
	\begin{align*}
		K_1&=\overline{\{(0,y,z) \in E_{p,\lambda}; 0\leq x,y<1\}},
	\end{align*}
	where the closure is taken with respect to the usual topology of $\mathbb{C}^3$. 
	
	An explicit formula of Bergman kernel $B$ on $E_{p,\lambda}$ is computed in \cite{BT15}:
	\begin{align}\label{eq:Bergman_kernel}
		B((x,y,z),\overline{(x,y,z)})&=\frac{\left((1-\nu_3)^{\lambda}-\nu_2 \right)^{\frac{1}{p}-3}{\nu_1}^2(p-1)(\lambda(p-1)+p) }{(1-\nu_3)^{2-2\lambda}\pi^3 p^2\left(\nu_1-((1-\nu_3)^{\lambda}-\nu_2)^{1/p} \right)^4 }\\\nonumber
		&+\frac{(1-\nu_3)^{\lambda-2}\left((1-\nu_3)^{\lambda}-\nu_2 \right)^{\frac{1}{p}-3}{\nu_1}^2(p-1)(\lambda-1)\nu_2p }{\pi^3 p^2\left(\nu_1-((1-\nu_3)^{\lambda}-\nu_2)^{1/p} \right)^4 }\\\nonumber
		&+\frac{\left((1-\nu_3)^{\lambda}-\nu_2 \right)^{\frac{3}{p}-3}(p+1)\left((1-\nu_3)^{\lambda}(\lambda+\lambda p +p)+(\lambda-1)\nu_2 p \right) }{(1-\nu_3)^{2-\lambda}\pi^3 p^2\left(\nu_1-((1-\nu_3)^{\lambda}-\nu_2)^{1/p} \right)^4 }\\\nonumber
		&-\frac{\left((1-\nu_3)^{\lambda}-\nu_2 \right)^{\frac{2}{p}-3}2\nu_1\left((1-\nu_3)^{\lambda}(\lambda(p^2-2)+p^2)+(\lambda-1)\nu_2 p^2 \right) }{(1-\nu_3)^{2-\lambda}\pi^3 p^2\left(\nu_1-((1-\nu_3)^{\lambda}-\nu_2)^{1/p} \right)^4 },\nonumber
	\end{align}
	where we set $\nu_1:=x\overline{x}$, $\nu_2:=y\overline{y}$ and $\nu_3:=z\overline{z}$.
	
	We write
	\[ a=1-\nu_3, \quad b= (1-\nu_3)^\lambda - \nu_2, \quad c= ((1-\nu_3)^\lambda - \nu_2)^{1/p}-\nu_1 . \]
	Then
	\begin{align} B &=\frac{b^{\frac{1}{p}-3}{\nu_1}^2(p-1)(\lambda(p-1)+p) }{a^{2-2\lambda}\pi^3 p^2 c^4 }
		+\frac{a^{\lambda-2} b^{\frac{1}{p}-3}{\nu_1}^2(p-1)(\lambda-1)\nu_2p }{\pi^3 p^2 c^4 }\\\nonumber
		&+\frac{b^{\frac{3}{p}-3}(p+1)\left(a^{\lambda}(\lambda+\lambda p +p)+(\lambda-1)\nu_2 p \right) }{a^{2-\lambda}\pi^3 p^2 c^4 }-\frac{b^{\frac{2}{p}-3}2\nu_1\left(a^{\lambda}(\lambda(p^2-2)+p^2)+(\lambda-1)\nu_2 p^2 \right) }{a^{2-\lambda}\pi^3 p^2 c^4 }.\nonumber
	\end{align}
	
	Write $D=a^{2} c^4$ and 
	\begin{align*} N&= {a^{2 \lambda} b^{\frac{1}{p}-3}{\nu_1}^2(p-1)(\lambda(p-1)+p) }
		+{a^{\lambda} b^{\frac{1}{p}-3}{\nu_1}^2(p-1)(\lambda-1)\nu_2p } \\
		& +{a^\lambda b^{\frac{3}{p}-3}(p+1)\left(a^{\lambda}(\lambda+\lambda p +p)+(\lambda-1)\nu_2 p \right) }\\ &-{a^\lambda b^{\frac{2}{p}-3}2\nu_1\left(a^{\lambda}(\lambda(p^2-2)+p^2)+(\lambda-1)\nu_2 p^2 \right) }. \end{align*}
	Then \begin{equation} \label{bfr} B= \frac N {\pi^3 p^2 D}. \end{equation}

	Write
	\begin{align*}  N_1 &= a^{2 \lambda} b^{\frac 1 p -3} \nu_1^2,  &  N_2 &= a^{\lambda} b^{\frac 1 p -3} \nu_1^2 \nu_2, &  N_3&  = a^{2 \lambda} b^{\frac 3 p -3},\\  N_4 &= a^{\lambda} b^{\frac 3 p -3} \nu_2, & N_5& = a^{2 \lambda} b^{\frac 2 p -3} \nu_1,&   N_6 &= a^{\lambda} b^{\frac 2 p -3} \nu_1\nu_2 , \\ u_1 &= (p-1)( \lambda (p-1) +p), & u_2 &= p(p-1)(\lambda -1), & u_3 &= (p+1)(\lambda + \lambda p +p), \\ u_4 &= p(p+1)(\lambda -1) , & u_5 &= -2 ( \lambda (p^2-2) + p^2) , & u_6 &= -2 (\lambda -1) p^2. \end{align*}
	Then
	\[  N= \sum_{i=1}^6 u_i N_i. \]
	Note that we have
	\[ u_1+u_3+u_5=6 \lambda \quad \text{ and } \quad u_2+u_4+u_6=0 .\]
	
	From the description of the Bergman kernel, we can check the pseudoconvexity of $E_{p,\lambda}$ for each $p,\lambda>0$. 
	
	\begin{prop}
		$E_{p,\lambda}$ is a pseudoconvex domain for each $p,\lambda>0$.
	\end{prop}
	\begin{proof}\!\!\footnote{This proof is suggested by an anonymous referee and replaces our original proof. We are grateful to the referee.}
		To show that $u=u_{p,\lambda}:=\left(|x|^{2p}+|y|^2 \right)^{\frac{1}{\lambda}}+|z|^2$ is a (bounded) plurisubharmonic exhaustion function of $E_{p,\lambda}$, it suffices to show that $v=v_{p,\lambda}:=\left(|x|^{2p}+|y|^2 \right)^{\frac{1}{\lambda}}$ is plurisubharmonic. To this end, consider 
		\begin{equation*}
			\log v = \frac{1}{\lambda}\log \left(e^{\psi_1}+e^{\psi_2} \right), \qquad \text{ where } \psi_1:=2p\log |x| \quad \text{ and } \quad  \psi_2:=2\log |y|.
		\end{equation*}
		Now the plurisubharmonicity of $\log v$ follows from the fact that $\log \left(e^{\psi_1}+e^{\psi_2} \right)$ is always plurisubharmonic whenever $\psi_1$ and $\psi_2$ are plurisubharmonic, since we have
		\begin{align*}
			&\frac{\partial^2}{\partial z \partial \overline{z} }\log \left(e^{\psi_1}+e^{\psi_2} \right)\\
			&=
			\frac{1}{\left(e^{\psi_1}+e^{\psi_2} \right)^2}\left(e^{\psi_1+\psi_2}\left(\frac{\partial \psi_1}{\partial z}-\frac{\partial \psi_2}{\partial z} \right)^2+e^{\psi_1}\frac{\partial^2 \psi_1}{\partial z \partial \overline{z}}+e^{\psi_2}\frac{\partial^2 \psi_2}{\partial z \partial \overline{z}}   \right)\geq 0.
		\end{align*}
		From the plurisubharmonicity of $\log v$ it follows that $v=e^{\log v}$ is plurisubharmonic, as desired. 
	\end{proof}

	We are interested in behaviours of the metric and curvature components on the compact set  
	$K_1=\overline{\{(0,y,z) \in E_{\lambda,p}; 0\leq y,z<1\}}$. In what follows, we compute those components.
	
	Recall the formula for the components of the Bergman metric \[ g_{i\overline j} = \frac{\partial^2 \log B}{\partial z_i \partial \overline{z_j}} , \quad i,j =1,2,3, \]
	where we set $(z_1,z_2,z_3)=(x,y,z)$. For $i=1,2,3$, we write
	\[ \partial_i= \frac{ \partial} {\partial z_i} \quad \text{ and } \quad \overline {\partial}_i= \frac{ \partial} {\partial \overline{z_i}}. \]
	
	\begin{prop}\label{prop:metric} Each component of the Bergman metric $g_{i\overline{j}}$ at $(0,y,z) \in E_{p,\lambda}, 0\leq y,z<1$, is given as follows:
		\begin{align*}
			g_{1\overline{1}}&= \frac 1 c \cdot \frac{u_5  + u_6 \delta}{u_3 + u_4 \delta} + \frac 4 c, \\
			g_{2\overline{2}} &= \frac {a^\lambda}{b^2} \left (\frac 1 p +3 \right ) + \frac {a^\lambda}{b^2} \cdot \frac{u_3 u_4 (1 -\delta)^2}{(u_3 + u_4 \delta)^2} , \\ 
			g_{2\overline{3}} = g_{3\overline{2}} &=  \frac {\lambda y z}{a^{1-\lambda} b^2} \cdot \left ( \frac 1 p +3 \right ) + \frac {\lambda y z}{a^{1-\lambda} b^2} \cdot \frac { u_3 u_4 (1-\delta)^2 }{ (u_3 + u_4 \delta)^2}  ,\\
			g_{3\overline{3}}& = \frac{1+\delta(\lambda z^2-1)}{ a^{2-2\lambda} b^2} \cdot \frac \lambda p+  \frac {\delta^2(2-2\lambda) + \delta(2\lambda^2 z^2 -4)+\lambda+2}{a^{2-2\lambda} b^2}\\&+ \frac {\lambda \delta}{a^{2-2\lambda} b^2} \cdot \frac {u_3 u_4 (1+\delta^2) (1+\lambda z^2) + u_4^2 \delta (1+(\lambda z^2-1) \delta +\delta^2) +u_3^2 (1+\lambda z^2)}{(u_3+u_4 \delta)^2},
			\\
			g_{i\overline{j}}& =0  \ \text{  otherwise,}
		\end{align*}
		where we write $\delta:= y^2/a^\lambda=y^2/(1-z^2)^\lambda$.\end{prop}
	
	\begin{proof}
		All the formulas for $g_{i \overline j}$ are obtained from direct computations.
		For example, since  
		\begin{align*}
			\op_1 D&=-4 a^2 c^3x, &
			\op_1 N_1 &= 2 a^{2 \lambda} b^{\frac 1 p -3} \nu_1 x,& 
			\op_1 N_2 &= 2 a^{\lambda} b^{\frac 1 p -3} \nu_1 x \nu_2 , \\
			\op_1 N_3 &=0 ,\qquad \op_1 N_4 =0 ,& 
			\op_1 N_5 &=  a^{2 \lambda} b^{\frac 2 p -3}  x,& 
			\op_1 N_6 &= a^{\lambda} b^{\frac 2 p -3}  x \nu_2  ,
		\end{align*}
		and
		\begin{align*}
			\p_1 \op_1 D&=-4 a^2 c^3 + 12 a^2 c^2 \nu_1, & 
			\p_1 \op_1 N_1 &= 4 a^{2 \lambda} b^{\frac 1 p -3} \nu_1,& 
			\p_1 \op_1 N_2 &= 4 a^{\lambda} b^{\frac 1 p -3} \nu_1 \nu_2 , \\
			\p_1 \op_1 N_3 &=0, \quad
			\p_1 \op_1 N_4 =0,& 
			\p_1 \op_1 N_5 &= a^{2 \lambda} b^{\frac 2 p -3} ,& 
			\p_1 \op_1 N_6 &= a^{\lambda} b^{\frac 2 p -3} \nu_2 , 
		\end{align*}
		we have
		\begin{align*} g_{1 \overline{1}} &= \frac{N (\partial_1 \overline{\partial}_1 N) - (\partial_1 N)( \overline{\partial}_1 N)}{N^2} - \frac{D (\partial_1 \overline{\partial}_1 D) - (\partial_1 D)( \overline{\partial}_1 D)}{D^2} \\ & \rsa \frac{\partial_1 \overline{\partial}_1 N}{N} - \frac{\partial_1 \overline{\partial}_1 D}{D} = \frac {u_5 a^{2\lambda}b^{\frac 2 p -3} + u_6 a^\lambda b^{\frac 2 p -3} y^2}{u_3 a^{2 \lambda} b^{\frac 3 p -3} +u_4 a^{\lambda} b^{\frac 3 p -3}y^2} + \frac{4 a^2 c^3}{a^2 c^4} \\ &\phantom{\rsa} =\frac 1 c \cdot \frac{u_5  + u_6 \delta}{u_3 + u_4 \delta} + \frac 4 c,  
		\end{align*}
		where we use $c=b^{\frac 1 p}$ at $(0,y,z)$. 
		
		The other $g_{i \overline{j}}$ can be computed similarly, and we omit the details.
	\end{proof}
	
	\begin{rema}
		When $(0,y,z)$ approaches the boundary of $K_1$, we find that the limits of the metric components and those of curvature components cannot be determined. However, using $\delta$ introduced in the above proposition, we will be able to control the limit behaviors. 
	\end{rema}
	
	Write 
	\begin{equation} \label{g-com} g_{1 \overline 1} = \frac 1 c \cdot A_{1}, \quad g_{2 \overline 2} = \frac {a^\lambda}{b^2} \cdot A_{2}, \quad
		g_{2 \overline 3} = \frac { \lambda yz }{a^{1-\lambda} b^2}\cdot A_2, \quad g_{3 \overline 3} = \frac 1 {a^{2-2\lambda} b^2} \cdot A_3, \end{equation}
	where 
	\begin{align*}
		A_1 & = \frac{u_5  + u_6 \delta}{u_3 + u_4 \delta} +  4 ,\qquad
		A_2  = \frac 1 p +3 + \frac{u_3 u_4 (1 -\delta)^2}{(u_3 + u_4 \delta)^2},\\
		A_3 & =  (1+\delta(\lambda z^2-1)) \cdot \frac \lambda p+ \delta^2(2-2\lambda) + \delta(2\lambda^2 z^2 -4)+\lambda+2\\&+ \lambda \delta \cdot \frac {u_3 u_4 (1+\delta^2) (1+\lambda z^2) + u_4^2 \delta (1+(\lambda z^2-1) \delta +\delta^2) +u_3^2 (1+\lambda z^2)}{(u_3+u_4 \delta)^2}.
	\end{align*}
	Then
	\begin{equation} \label{eq-det} g_{2 \overline 2} g_{3 \overline 3} - g_{2 \overline 3} g_{3 \overline 2} = \frac 1 {a^{2-3\lambda} b^4} {A_2 (A_3 - \lambda^2 \delta  z^2 A_2)}  = \frac {1-\delta} {a^{2-3\lambda} b^4} \cdot A_2 A_4= \frac{A_2A_4}{a^{2-2\lambda}b^3}, \end{equation} where we put $A_4:= (A_3-\lambda^2 \delta z^2 A_2)/(1-\delta)$ and use $1-\delta= b/a^{\lambda}$.
	More explicitly, we have
	\[A_4=\frac{ \delta ^2 p^2
		(r-2) (r-1)+\delta  p (r-1) (4 pr+4p+3 r)+p^2 r^2+3 p^2 r+2 p^2+2
		p r^2+3 p r+r^2}{p( \delta 
		p (r-1)+p r+p+r)}.\]
	Note that $0 \le \delta <1$. Furthermore, as $(0,y,z) \in E_{p,\lambda}$ approaches the boundary, we have $\delta \rightarrow 1^-$. One sees that \begin{equation} \label{limA} \lim_{\delta \rightarrow 1^-} A_1 =\frac{4(2+p)}{1+2p}, \qquad \lim_{\delta \rightarrow 1^-} A_2 =3+ \frac 1 p 
		\quad \text{ and } \quad  \lim_{\delta \rightarrow 1^-}A_4=\lambda \left ( 3 +\frac 1 p \right ) .\end{equation}
	
	\begin{lemma}\label{lem:det/kernel}
		At $(0,y,z) \in E_{p,\lambda}$, $0\le y,z <1$, the ratio $\displaystyle{\frac{\det g_B}B}$ is bounded.
	\end{lemma}
	
	\begin{proof}
		From \eqref{bfr}, \eqref{g-com} and \eqref{eq-det}, we obtain \begin{align*}
			\frac{\det g_B}{B}&= \frac{ \frac 1 c A_1 \frac{A_2A_4}{a^{2-2\lambda}b^3} } {\frac N {\pi^3 p^2 D}} = \frac{\pi^3 p^2 A_1 A_2 A_4 a^2 c^4}{c a^{2-2\lambda}b^3 \cdot  a^\lambda b^{\frac 3 p -3}(p+1) \left ( a^\lambda (\lambda + \lambda p + p) + (\lambda -1)y^2 p \right )} \\ &  = \frac{\pi^3 p^2 A_1 A_2 A_4}{(p+1) \left (  (\lambda + \lambda p + p) + (\lambda -1) p \delta \right )},
		\end{align*}
		which is bounded.
	\end{proof}

	\begin{prop}\label{prop:inverse-metric} 
		The inverse metric of the Bergman metric $g_{i\overline{j}}$ at $(0,y,z)\in E_{p,\lambda}$, $0\leq y,z<1$, are given as follows:
		\begin{align*}
			g^{1\overline{1}}&= \frac c {A_1}, &
			g^{2\overline{2}}&=\frac{b^2} {a^\lambda} \cdot \frac { A_3}{(1-\delta) A_2 A_4}=\frac {bA_3}{A_2A_4},\\ 
			g^{2\overline{3}}=g^{3\overline 2} &= - \frac {\lambda yza^{1-2\lambda} b^2} {(1-\delta)A_4}=-\frac{\lambda y z a^{1-\lambda} b}{A_4}, &
			g^{3\overline{3}}&=\frac{a^{2-2\lambda} b^2}{(1-\delta)A_4}=\frac{a^{2-\lambda}b}{A_4} , \\ g^{i\overline{j}}& =0  \ \text{  otherwise.}
		\end{align*}	
	\end{prop}
	
	\begin{proof}
		The formulas are obtained by taking the inverse matrix of the $3 \times 3$ matrix $(g_{i \overline{j}})_{i,j=1,2,3}$ calculated in Proposition \ref{prop:metric}. In particular, the determinant of the $2 \times 2$ block $(g_{i \overline{j}})_{i,j=2,3}$ is computed in \eqref{eq-det}. Also recall $1-\delta= b/a^\lambda$.
	\end{proof}

	Through direct computations, we obtain the following for $(0,y,z) \in K_1$:
	\begin{table}[H]
		\fbox{\parbox{\textwidth}{
				\begin{align*}
					\p_1 g_{2 \overline 1} =\p_2 g_{1 \overline 1}  =  \op_1 g_{1 \overline 2} =\op_2 g_{1 \overline 1}  =
					& \frac y { bc } G_1, \\ 
					\p_1 g_{3 \overline 1} = \p_3 g_{1 \overline 1} = \op_1 g_{1 \overline 3} = \op_3 g_{1 \overline 1} =
					& \frac z {a^{1-\lambda} bc} G_2 , \\
					\p_2 g_{2 \overline 2} = \op_2 g_{2 \overline 2} =
					& \frac {ya^\lambda} { b^3} G_3,\\
					\p_2 g_{2 \overline 3} =  \op_2 g_{3 \overline 2} = & \frac { y^2 z}{a^{1-\lambda} b^3} G_4, \\
					\p_2 g_{3 \overline 2} = \p_3 g_{2 \overline 2} = \op_2 g_{2 \overline 3} = \op_3 g_{2 \overline 2} =& \frac { y^2 z}{a^{1-\lambda} b^3} G_5, \\
					\p_2 g_{3 \overline 3} = \p_3 g_{2 \overline 3} =  \op_2 g_{3 \overline 3} = \op_3 g_{3 \overline 2} = & \frac { y z^2}{a^{2-2\lambda} b^3} G_6, \\
					\p_3 g_{3 \overline 2} =  \op_3 g_{2 \overline 3} = & \frac { y z^2}{a^{2-2\lambda} b^3} G_7, \\
					\p_3 g_{3 \overline 3} = \op_3 g_{3 \overline 3} =& \frac { z}{a^{3-3 \lambda} b^3} G_8, \\
					\p_i g_{j \overline k} = \op_i g_{j \overline k} = &0 \  \text{ otherwise}.
				\end{align*}
		}}
		\caption{Formulas for $\p_i g_{j \overline k}$}
		\label{Gs}
	\end{table}
	Here $G_i$ are set to be the remaining factors after pulling out the factors involving $a,b,c,y,z$. Explicitly, we have 
	\begin{align*}
		G_1  =& \frac {4 }{p } - \frac {(u_5+u_6 \delta) ( (2p-3) u_4 \delta + 3(p-1)u_3 + p u_4)} {p (u_3+ \delta u_4)^2}  + \frac {2(p-1) u_6 \delta+ (3p-2)u_5 + p u_6} {p(u_3 +u_4 \delta)},\\ 
		G_2 = & \frac {4 \lambda }{ p } + \frac {\lambda }{p} \cdot \frac {u_5+ u_6 \delta}{u_3 + u_4 \delta} - \frac { \lambda  \delta (1-\delta)  ( u_4 u_5-u_3 u_6)}{(u_3 + u_4 \delta)^2} . 
	\end{align*}
	For simplicity, we do not present expressions for the other $G_i$'s. Since $u_3+u_4\delta >0$, one can see that $G_i$ are bounded for $i =1, 2, \dots , 8$ as $\delta \rightarrow 1^-$.  
	
	\begin{lemma}
		We have
		\[ G_4=\lambda G_3.\]
		If we define $F_1$ and $F_2$ by
		\[ F_1 :=  \frac{z^2}{1-\delta} \left (G_6- \lambda \delta G_5 \right ) \quad \text{ and } \quad  F_2:=\frac 1{1-\delta}\left ( G_8-\lambda \delta z^2 G_7 \right ) , \] then 
		\[  \lim_{\delta \rightarrow 1^-} F_1 = \lambda \left (3+\frac 1 p \right ) \quad \text{ and } \quad \lim_{\delta \rightarrow 1^-} F_2 = \frac {2\lambda^2 (1+3p)}p. \]
	\end{lemma}
	
	\begin{proof}
		We verify the identities through direct computations with help of a computer algebra system.  
	\end{proof}
	
	Similarly, we obtain
	\begin{table}[H]
		\fbox{\parbox{\textwidth}{
				\begin{align*}
					\p_1 \op_1 g_{1 \overline 1} =& \frac 1 {c^2} H_1  , \\
					\p_1 \op_1 g_{2 \overline 2}=\p_1 \op_2 g_{2 \overline 1}=\p_2 \op_1 g_{1 \overline 2}= \p_2 \op_2 g_{1 \overline 1} =&  \frac {a^\lambda} {b^2 c} H_2, \\
					\p_1 \op_1 g_{2 \overline 3}= \p_1 \op_3 g_{2 \overline 1}= \p_2 \op_1 g_{1 \overline 3} =\p_2 \op_3 g_{1 \overline 1} =&
					\p_1 \op_1 g_{3 \overline 2} =\p_1 \op_2 g_{3 \overline 1} =\p_3 \op_1 g_{1 \overline 2} =\p_3 \op_2 g_{1 \overline 1} = \frac{y z}{a^{1-\lambda} b^2 c} H_3 , \\
					\p_1 \op_1 g_{3 \overline 3} =\p_1 \op_3 g_{3 \overline 1} =\p_3 \op_1 g_{1 \overline 3} =\p_3 \op_3 g_{1 \overline 1} =&  \frac 1 { a^{2-2\lambda} b^2 c} H_4, \\
					\p_2 \op_2 g_{2 \overline 2} =& \frac{a^{2\lambda}}{b^4} H_5  , \\
					\p_2 \op_2 g_{2 \overline 3} =\p_2 \op_3 g_{2 \overline 2}=\p_2 \op_2 g_{3 \overline 2} =\p_3 \op_2 g_{2 \overline 2}   =& \frac { y z }{a^{1-2\lambda} b^4} H_6  , \\
					\p_2 \op_2 g_{3 \overline 3} =\p_2 \op_3 g_{3 \overline 2} =\p_3 \op_2 g_{2 \overline 3} =\p_3 \op_3 g_{2 \overline 2} =& \frac 1 {a^{2-3 \lambda} b^4} H_7 , \\
					\p_2 \op_3 g_{2 \overline 3} =\p_3 \op_2 g_{3 \overline 2} =& \frac{y^2 z^2}{a^{2-2\lambda} b^4} H_8 , \\
					\p_2 \op_3 g_{3 \overline 3} =\p_3 \op_3 g_{2 \overline 3} =\p_3 \op_2 g_{3 \overline 3} =\p_3 \op_3 g_{3 \overline 2} =& \frac {yz}{a^{3-3\lambda} b^4} H_9, \\
					\p_3 \op_3 g_{3 \overline 3} =&\frac 1 { a^{4-4\lambda} b^4} H_{10}  , \\
					\p_i \op_j g_{k \overline l} =& 0 \ \text{ otherwise}. 
				\end{align*}
		}}
		\caption{Formulas for $\p_i \op_j g_{k \overline l}$}
		\label{Hs}
	\end{table}
	Here $H_i$ are the remaining factors; in particular, we have \[ H_1 = 8 +   4 \cdot \frac {u_1+u_2 \delta}{u_3+u_4 \delta} -  2 \cdot \frac {(u_5+u_6 \delta)^2} { (u_3 + u_4 \delta)^2}.\]
	We do not present explicit expressions for the other $H_i$'s. Using $0 \le \delta <1$ and $u_3+u_4\delta >0$, one can check that $H_i$ are bounded for $i =1, 2, \dots , 10$ as $\delta \rightarrow 1^-$. 
	
	\begin{prop}\label{prop:curv_components} Each curvature components of the Bergman metric at $(0,y,z)\in E_{p, \lambda}$, $0\leq y,z< 1$, is given by
		\begin{align*}
			R_{1\overline{1}1\overline{1}}&=\frac 1 {c^2} {(-H_1)} = \frac 1 {c^2} \cdot \widetilde H_1,\\
			R_{1\overline{1}2\overline{2}}&=R_{2\overline{1}1\overline{2}}=R_{1\overline{2}2\overline{1}}=R_{2\overline{2}1\overline{1}}=\frac{a^\lambda} {b^2 c} \cdot  \left (  - {H_2}+ \frac{\delta{G_1}^2}{A_1} \right )=\frac{a^\lambda} {b^2 c} \cdot  \widetilde H_2 ,\\
			R_{1\overline{1}2\overline{3}}&=R_{1\overline{1}3\overline{2}}=R_{2\overline{1}1\overline{3}}=R_{1\overline{2}3\overline{1}}=R_{1\overline{3}2\overline{1}}=R_{2\overline{3}1\overline{1}}=R_{3\overline{1}1\overline{2}}=R_{3\overline{2}1\overline{1}}\\ &=\frac {yz a^{\lambda -1}} {b^2c} \cdot \left (  - H_3 + \frac {G_1 G_2}{A_1} \right )=\frac {yz a^{\lambda -1}} {b^2c} \cdot \widetilde H_3 ,\\
			R_{1\overline{1}3\overline{3}}&=R_{1\overline{3}3\overline{1}}=R_{3\overline{1}1\overline{3}}=R_{3\overline{3}1\overline{1}}=\frac { a^{2\lambda -2}} {b^2c} \cdot \left ( - H_4+ \frac {z^2 G_2^2}{A_1} \right ) =\frac { a^{2\lambda -2}} {b^2c} \cdot \widetilde H_4 ,\\
			R_{2\overline{2}2\overline{2}}&=\frac{a^{2\lambda}}{b^4} \cdot \left ( - H_5 + \frac{\delta  G_3^2}{A_2} \right )  =\frac{a^{2\lambda}}{b^4} \cdot \widetilde H_5  ,\\
			R_{2\overline{2}2\overline{3}}&=R_{2\overline{2}3\overline{2}}=R_{2\overline{3}2\overline{2}}=R_{3\overline{2}2\overline{2}}=\frac{yza^{2\lambda -1}}{b^4} \cdot \left ( - H_6 + \frac{\delta G_3G_5} {A_2 } \right ) =\frac{yza^{2\lambda -1}}{b^4} \cdot \widetilde H_6,\\
			R_{2\overline{2}3\overline{3}}&=R_{2\overline{3}3\overline{2}}=R_{3\overline{2}2\overline{3}}=R_{3\overline{3}2\overline{2}} \\&= \frac {a^{3\lambda -2}}{b^4} \cdot  \left ( - H_7+\frac{\delta^2 z^2 G_5^2}{A_2} + \frac{\delta (1-\delta) F_1^2}{A_4} \right )= \frac {a^{3\lambda -2}}{b^4} \cdot \widetilde H_7,\\
			R_{2\overline{3}2\overline{3}}&=R_{3\overline{2}3\overline{2}}= \frac {a^{2\lambda -2}y^2 z^2}{b^4} \cdot  \left ( - H_8 + \frac{G_3G_7}{A_2}\right )= \frac {a^{2\lambda -2}y^2 z^2}{b^4} \cdot \widetilde H_8,\\
			R_{2\overline{3}3\overline{3}}&=R_{3\overline{2}3\overline{3}}=R_{3\overline{3}2\overline{3}}=R_{3\overline{3}3\overline{2}}\\&= \frac {a^{3\lambda -3}y z}{b^4} \cdot  \left ( - H_9 + \frac{\delta z^2 G_5 G_7}{A_2} + \frac{(1-\delta) F_1F_2}{A_4} \right )= \frac {a^{3\lambda -3}y z}{b^4} \cdot \widetilde H_9,\\
			R_{3\overline{3}3\overline{3}}&=\frac {a^{4\lambda -4}}{b^4} \cdot  \left ( - H_{10} + \frac{\delta z^4 G_7^2}{A_2}+\frac {z^2 (1-\delta) F_2^2}{A_4} \right )=\frac {a^{4\lambda -4}}{b^4} \cdot \widetilde H_{10},\\
			R_{i\overline{j}k\overline{l}}&=0 \ \text{  otherwise,}
		\end{align*}
		where we define $\widetilde H_i$ for $i=1, 2, \cdots , 10$ for later use.
	\end{prop}
	
	\begin{proof}
		Recall that the components of curvature tensor $R$ associated with $g$ is given by 
		\begin{equation*}\label{eq:curv}
			R_{i\overline{j}k\overline{l}}=-\partial_k \overline \partial_l g_{i\overline{j}}+ \sum_{p,q=1}^{3}g^{q\overline{p}}(\partial_k g_{i\overline{p}})( \overline \partial_l  g_{q\overline{j}}).
		\end{equation*}
		Thus the results follow from Tables \ref{Gs} and \ref{Hs} and Proposition \ref{prop:inverse-metric}. 
	\end{proof}
	
	\begin{lemma} \label{lem-tH}
		We have
		\[ \widetilde H_3 = \lambda \widetilde H_2,  \quad  \widetilde H_6 = \lambda \widetilde H_5, \quad  \widetilde H_8 = \lambda \widetilde H_6 \quad \text{and} \quad \widetilde H_9=2\lambda \widetilde H_7- \lambda^2 \delta z^2 \widetilde H_6. \]
		If we define
		\begin{align*}
			\widetilde F_1 &:=\frac 1 {1-\delta} \left (
			\widetilde H_4 - \lambda \delta z^2 \widetilde H_3 \right ),\qquad \widetilde F_2 :=\frac 1 {1-\delta} \left ( \widetilde H_7 - \lambda \delta z^2 \widetilde H_6 \right ) , \\
			\widetilde F_3 &= \frac 1 {(1-\delta)^2} \left (\widetilde H_{10} - 4 \lambda^2 \delta z^2 \widetilde H_7 + 3 \lambda^3 \delta^2 z^4 \widetilde H_6 \right ) , \end{align*}
		then 
		\begin{equation} \label{limF}  \lim_{\delta \rightarrow 1^-} \widetilde F_1 = - \frac{4 \lambda (2+p)}{p (1+2p)}, \quad \lim_{\delta \rightarrow 1^-} \widetilde F_2 =-\lambda \left (3+\frac 1 p\right )  \quad \text{and} \quad \lim_{\delta \rightarrow 1^-} \widetilde F_3= -2\lambda^2 \left ( 3+ \frac 1 p \right ).\end{equation} \end{lemma}

	\begin{proof}
		The identities are verified through direct computations and can be checked by a computer algebra system.  
	\end{proof}

	In order to see cancellations of factors involving $a, b, c$ in the holomorphic sectional curvature, we apply the Gram--Schmidt process to determine an orthonormal frame $X,Y,Z$ instead of using the global coordinate vector fields $\frac{\partial}{\partial z_i},i=1,2,3$. 
	Indeed, let $g$ be any Hermitian metric, and take the first unit vector field 
	\begin{equation}\label{eq:orthonormal_X}
		X=\frac{\partial_1}{\sqrt{g_{1\overline{1}}}}.
	\end{equation}
	Write $k_1:= \frac 1{\sqrt{g_{1\overline{1}}}}$ so that $X= k_1 \partial _1$. Then a vector field $\tilde{Y}$ which is orthogonal to $X$ is given by 
	\begin{equation*}
		\tilde{Y}=\frac{\partial_2}{\sqrt{g_{2\overline{2}}}}-g \left (\frac{\partial_2}{\sqrt{g_{2\overline{2}}}},X \right )X=a_1 \partial_1+a_2 \partial_2,
	\end{equation*}
	where we put \[a_1:=-\frac{g_{2\overline{1}}}{g_{1\overline{1}}\sqrt{g_{2\overline{2}}}} \quad \text{ and } \quad a_2:=\frac{1}{\sqrt{g_{2\overline{2}}}}.\]
	
	Since $g(\tilde Y,\tilde Y)=a_1 \overline{a_1}g_{1\overline{1}}+a_1\overline{a_2}g_{1\overline{2}}+a_2\overline{a_1}g_{2\overline{1}}+a_2\overline{a_2}g_{2\overline{2}}$, we take 
	\begin{equation}\label{eq:orthonormal_Y}
		Y=\frac{\tilde{Y}}{\sqrt{g(\tilde{Y},\tilde{Y})}}=\frac{a_1 \partial_1+a_2 \partial_2}{\sqrt{a_1 \overline{a_1}g_{1\overline{1}}+a_1\overline{a_2}g_{1\overline{2}}+a_2\overline{a_1}g_{2\overline{1}}+a_2\overline{a_2}g_{2\overline{2}}}}=t_1\partial_1+t_2\partial_2,
	\end{equation}
	where we put \begin{equation}\label{eq:t_i}
		t_i:=\frac{a_i}{{\sqrt{a_1 \overline{a_1}g_{1\overline{1}}+a_1\overline{a_2}g_{1\overline{2}}+a_2\overline{a_1}g_{2\overline{1}}+a_2\overline{a_2}g_{2\overline{2}}}}}, \quad i=1,2.
	\end{equation}
	
	Similarly, consider 
	\begin{align*}
		\tilde{Z}=p_1\partial_1+p_2\partial_2+p_3\partial_3,
	\end{align*}
	where
	\begin{align*}
		p_1&:=-\frac{{g_{3\overline{1}}}}{g_{1\overline{1}}\sqrt{g_{3\overline{3}}}}-\frac{t_1}{\sqrt{g_{3\overline{3}}}}(t_1g_{3\overline{1}}+t_2g_{3\overline{2}}), \\
		p_2&:=-\frac{{t_2}}{\sqrt{g_{3\overline{3}}}}(t_1g_{3\overline{1}}+t_2g_{3\overline{2}}),\qquad p_3:=\frac{1}{\sqrt{g_{3\overline{3}}}}.
	\end{align*}
	Normalizing $\tilde{Z}$ yields
	\begin{align}\label{eq:orthonormal_Z}
		Z&=s_1\partial_1+s_2\partial_2+s_3\partial_3,
	\end{align}
	where 
	\begin{align*}
		s_i&:=\frac{p_i}{\sqrt{ \sum^{3}_{k,l=1} p_{k}p_{l}g_{k\overline{l}}}},\quad i=1,2,3.
	\end{align*}
	
	These $X,Y,Z$ are used in the following proposition which is the main result of this section.  
	
	\begin{prop}\label{prop:components}
		At $(0,y,z)\in E_{p, \lambda}, 0\leq y,z< 1$, the components of the holomorphic sectional curvature $R$ are given by as follows.
		\begin{align*}
			H(X)&=R(X, \bar X, X, \bar X)=  \frac {\widetilde H_1}{A_1^2},&B(X,Y)&=R(X, \bar X, Y, \bar Y)= \frac {\widetilde H_2} {A_1 A_2},\\
			H(Y)&=R(Y, \bar Y, Y, \bar Y)= \frac {\widetilde H_5} {A_2^2},&B(X,Z)&=R(X, \bar X, Z, \bar Z)= \frac{\widetilde F_1}{A_1 A_4},\\ 
			H(Z)&=R(Z, \bar Z, Z, \bar Z)=\frac {\widetilde F_3}{A_4^2}, &B(Y,Z)&=R(Y, \bar Y, Z, \bar Z)=\frac{\widetilde F_2}{A_2 A_4},  \end{align*}
		\begin{align*}
			&R(X, \bar X, X, \bar Y)=R(Y, \bar Y, Y, \bar X)=R(Z, \bar Z, Z, \bar Y)=R(Y, \bar X,Y, \bar X)=0, \\ & R(X, \bar X, X, \bar Z)=R(Y, \bar Y, Y, \bar Z)=R(Z, \bar Z, Z, \bar X)=R(Z, \bar X,Z, \bar X)=0,\\
			&R(X,\bar{X},Y,\bar{Z})=R(Y,\bar{Y},X,\bar{Z})=R(Z,\bar{Z},X,\bar{Y})=R(Z, \bar Y,Z, \bar Y)=0.
		\end{align*}
		
	\end{prop}
	
	\begin{proof}
		All the identities follow from Proposition \ref{prop:curv_components} and Lemma \ref{lem-tH}. To illustrate the process, we compute $H(X)$, $B(X,Y)$ and $R(Y, \bar Y, Y, \bar Z)$. Computations of the other components are similar.
		
		Since $g_{2\overline{1}}=0$ and $g_{3\overline{1}}=0$, we have $a_1=0$, $t_1 =0$, $p_1=0$ and $s_1=0$ on $(0,y,z)$. On the other hand,  
		\[ t_2 = \frac{a_2}{ \sqrt{a_2 \overline{a_2} g_{2 \overline 2}} } = \frac 1 {\sqrt {g_{2 \overline 2}}} .\]
		Thus, using \eqref{g-com},  we obtain
		\begin{align*} H(Y) &= t_2^4 R_{2 \overline 2 2 \overline 2} = \frac {b^4}{a^{2 \lambda}} \frac 1 {A_2^2} \cdot \frac{a^{2\lambda}}{b^4} \widetilde H_5 = \frac {\widetilde H_5} {A_2^2}. \end{align*} 
		Similarly, 
		\begin{align*} B(X,Y) & = k_1^2 t_2^2 R_{1 \overline 1 2 \overline 2}= \frac 1 {g_{1 \overline 1}} \frac 1 {g_{2 \overline 2}} \cdot 
			\frac{a^\lambda} {b^2 c} \cdot \widetilde H_2 = \frac c {A_1} \frac {b^2}{a^\lambda A_2} \frac {a^\lambda} {b^2 c} \widetilde H_2   = \frac 1 {A_1 A_2}\widetilde H_2. \end{align*}
		
		To compute $R(Y, \bar Y, Y, \bar Z)$, first observe
		\[ s_2= -s_3 t_2^2 g_{3 \overline{2}} = -s_3 \frac {g_{3 \overline{2}}}{g_{2 \overline{2}}}= -s_3 \frac {\lambda y z}{a}. \]
		Thus it follows from  Proposition \ref{prop:curv_components} and Lemma \ref{lem-tH} that 
		\begin{align*}
			R(Y, \bar Y, Y, \bar Z) &=t_2^3 s_2 R_{2 \overline 2 2 \overline 2}+t_2^3s_3 R_{2 \overline 2 2 \overline 3}= t_2^3 \left ( - s_3 \frac {\lambda y z}{a} \right ) \frac{a^{2\lambda}}{b^4} \widetilde H_5+t_2^3s_3 \frac{yza^{2\lambda-1}}{b^4} \widetilde H_6\\ &= \frac {t_2^3 s_3 a^{2\lambda -1} yz}{b^4} \left ( -\lambda \widetilde H_5 + \widetilde H_6 \right ) =0. 
		\end{align*}
	\end{proof}
	
	\begin{coro}\label{prop:hsc} 
		The holomorphic sectional curvature near $\partial K_1$ is bounded for any $p, \lambda >0$.\end{coro}
	
	\begin{proof}
		The assertion follows from \eqref{limA} and \eqref{limF} and the fact that $G_i$ and $H_i$ are bounded as $\delta \rightarrow 1^-$.
	\end{proof}

	It is known \cite{GC19} that the  curvature tensor of the Bergman metric is bounded for $\lambda=1$ and $p >0$.  The following proposition tells us that the same is true for any $p,\lambda >0$.
	
	\begin{prop}\label{prop:bounded_geometry_bergman_step1} The curvature tensor of the Bergman metric on $E_{p,\lambda}$ is bounded for any $p,\lambda >0$. 
	\end{prop}
	
	\begin{proof}
		The curvature tensor can be explicitly expressed in terms of the holomorphic sectional curvature $H_{g_B}$. Using the invariance of the Bergman metric, it suffices to show $H_{g_B} \le C$ on $\partial K_1$ by some constant $C\in \mathbb{R}$. By Corollary \ref{prop:hsc}, we are done. 
	\end{proof}	
	
	\begin{coro}\label{cor:main}
		For any $p,\lambda >0$, there exist $C_0>0$ such that 
		\begin{equation*}
			\chi_{E_{p,\lambda}}(p;v) \leq C_0 \sqrt{\omega_{B}(v,v)} \qquad \text{ for all } v\in T'_{p}E_{p,\lambda}, \, p\in M,
		\end{equation*}
		and $C_1>0$ such that 
		\begin{equation*}
			\frac{1}{C_1}\omega_{KE}(v,v)\leq \omega_{B}(v,v) \leq C_1\omega_{KE}(v,v) \qquad \text{ for all } v\in T'E_{p,\lambda}.
		\end{equation*}
		\begin{proof}
			The assertion immediately follows from Proposition~\ref{prop:bounded_geometry_bergman_step1} and Lemma~\ref{lem:det/kernel}.
		\end{proof}	
		
	\end{coro}	
	
	\begin{rema}
		For the third statement of Theorem~\ref{thm:comparison}, in general, the holomorphic sectional curvature is not negatively pinched for $E_{p, \lambda}$. For example, when $\lambda=1$ and $p=1/5$, we have $\lim_{\delta \rightarrow 1^-} H(X) \approx 0.033 >0$. \end{rema}
	
	Lastly, we obtain interesting rigidity in the following proposition from direct computation of the Ricci curvature of the Bergman metric and we omit the proof. 
	\begin{prop}\label{prop:aysmp-Ric} 
		The Bergman metric $g_B$ on $E_{p,\lambda}$ is a K\"ahler--Einstein metric if and only if $\lambda=p=1$.
	\end{prop}

	\section{A lower bound of the integrated Carath\'eodory--Reiffen metric}
	In this last section, we prove the following theorem.
	
	\begin{mainthm}\label{thm:lower-bounded-c-metric}
		Let $(M,g)$ be a simply-connected complete noncompact $n$-dimensional K\"ahler manifold whose Riemannian sectional curvature $k$ of $g$ satisfies $ k \leq -a^2$ for some $a>0$. We denote by $d$ the geodesic distance on $M$, and by $\gamma_M$ the Carath\'eodory--Reiffen metric on $M$. For any $p\geq 2$, the following are true.  
		
		1. 	Let $f$ be a holomorphic function from $M$ to the unit disk $\mathbb{D}$ in $\mathbb{C}$. Then	\begin{equation}\label{eq:gradient-estimate2}	
			\int_{M}  \left | \int_{M} G(x,y)|\nabla f|^2(y)dy \right |^p  dx  \leq \left(\frac{p}{(2n-1)a} \right)^{p} \int_{M} |f(x)|^p \gamma_M(x;\nabla f(x))^{\frac{p}{2}}dx ,  
		\end{equation}
		where $G(x,y)$ is the minimal positive Green's function on $M$. 
		
		2. If the Riemannian sectional curvature $k$ of $g$ further satisfies $-b^2 \leq k$ for some $b>0$.	 Then there exists a constant $C(n)>0$, which only depends on $n$, such that for any holomorphic function $f$ from $M$ to the unit disk $\mathbb{D}$, we have  
		\begin{align}\label{eq:gradient-estimate}	
			& \int_{0}^{\infty}\int_{M}\left( \int_{M} t^{-n}\exp[-\frac{d(x,y)^2}{2t}-\frac{(2n-1)^2 b^2t}{8}-\frac{(2n-1)b d(x,y) }{2}](1+b d(x,y))|\nabla f|^2(y)dy\right)^p dx dt \\	\nonumber
			& \leq C(n) \left(\frac{2\pi p}{(2n-1)a} \right)^p  \int_{M} |f(x)|^p\gamma_M(x;\nabla f(x))^{\frac{p}{2}}dx.
		\end{align} 
	\end{mainthm}
	
	The inequalities~\eqref{eq:gradient-estimate2} and ~\eqref{eq:gradient-estimate} can be interpreted as integrated gradient estimates of bounded holomorphic functions. 
	
	Although the lemmas below are known, we prove them here for tracking explicit constants for the proof of Theorem~\ref*{thm:lower-bounded-c-metric}.
	
	Let $M$ be an $n$-dimensional complete noncompact, simply connected Riemannian manifold, and let $L^2(M)$ be the space of $L^2$-functions on $M$. Denote by $W^1(M)$ the Hilbert space consisting of $L^2$-functions whose gradient are also $L^2$, and by $W^1_{0}(M)$ the subspace in $W^1(M)$ which is the completion of the space $C^{\infty}_{0}(M)$ under $W^1(M)$-norm. When $M$ is complete, we have $W^1(M)=W^1_{0}(M)$.
	
	\begin{lemma}[{\cite[Poincar\'e inequality]{SchoenRYauST94}}] \label{eq:Poincare}
		Let $M$ be an $n$-dimensional complete noncompact, simply connected Riemannian manifold with sectional curvature $k\leq -a^2<0$. Then 
		\begin{equation}\label{eq:1.2}
			\int_{M} |u|^2 \leq \frac{4}{(n-1)^2a^2}\int_{M} |\nabla u|^2, \qquad u\in W^{1}_{0}(M).
		\end{equation}
		
		\begin{proof}
			Let $r(x)=d(p_0,x)$ be the distance function from a fixed point $p_0 \in M$. From the Rauch comparison theorem, we have
			\begin{equation}\label{eq:1.3}
				\triangle r \geq (n-1)a,
			\end{equation}
			where $a>0$. 
			
			Let $\Omega$ be the geodesic ball centered at $p_0$ with radius $R>0$ in $M$. From the Green's theorem, we have for every $u\in C^{\infty}_{0}(\Omega)$,
			\begin{equation*}
				\int_{\Omega}|u|^2 \triangle r - \int_{\Omega} \nabla(|u|^2)\cdot{\nabla r} = \int_{b\Omega}|u|^2 d\sigma =0,
			\end{equation*}
			where $d\sigma$ is the surface measure on $b\Omega$. We remark that $r$ may not be smooth at $p_0$, but we can apply the Green's theorem to $\Omega$ minus a small ball of radius $\epsilon>0$ around $p_0$ and let $\epsilon \rightarrow 0$. 
			From \eqref{eq:1.3} and $|\nabla r|=1$, we have
			\begin{align*}
				(n-1)a\lVert u \rVert^2\leq \int_{\Omega} |u|^2 \triangle r = \int_{\Omega} \nabla (|u|^2)\cdot \nabla r \leq \int_{\Omega} |\nabla (|u|^2) | \leq 2 \lVert u \rVert \, \lVert \nabla u \rVert.
			\end{align*}
			This gives
			\begin{equation*}
				\lVert u \rVert  \leq \frac{2}{(n-1)a} \lVert  \nabla u \rVert , \qquad  u\in C^{\infty}_{0}(\Omega).
			\end{equation*}
			Since $C^{\infty}_{0}(M)$ is dense in $W^{1}_{0}(M)$, we are done. 
		\end{proof}
		
	\end{lemma}
	
	Let $\triangle_0$ denote the Laplace--Beltrami operator. We use Mckean's estimate \cite{McKeanHP70} on the first eigenvalue of $\triangle_0$. 
	
	\begin{lemma}[{\cite[Mckean's estimate]{McKeanHP70}}] \label{Lem:Mckean}
		Let $M$ be an $n$-dimensional complete noncompact, simply-connected Riemannian manifold with sectional curvature $k\leq -a^2<0$. Then
		we have	\begin{equation}\label{eq:1.4}
			\lambda_1\geq \frac{(n-1)^2a^2}{4},
		\end{equation}
		where $\lambda_1$ is the smallest eigenvalue of $\triangle_0$.  	
		\begin{proof}
			From Lemma~\ref{eq:Poincare}, for every $u\in C^{\infty}_{0}(M)$, 
			\begin{equation*}
				(\triangle_0 u,u)=(du,du)=\int_{\Omega} |\nabla u |^2\geq \frac{(n-1)^2a^2}{4}\int_{\Omega} |u|^2.
			\end{equation*}
			The assertion follows. 
		\end{proof}
		
	\end{lemma}
	
	\begin{lemma}[{\cite[Cheng]{ChengShiuYuen93}}] \label{Lem:Cheng}
		Let $M$ be an $n$-dimensional Riemannian manifold. Consider the first eigenvalue for the Dirichlet problem $\lambda_1(M)>0$. Let $\Omega$ be a relatively compact domain of $M$ such that $b\Omega$ is smooth. Let $f\in C^{\infty}(M)$ and let $u$ be the solution of 
		\begin{align*}
			\begin{cases} \triangle u = \triangle f  & \text{on  } \Omega, \\
				u = 0 & \text{on  } b\Omega.  \end{cases}
		\end{align*}
		Then for any $p\geq 2$,
		\begin{equation}\label{eq:1.5}
			\int_{\Omega} |u|^p \leq C_p \int_{\Omega} |\nabla f|^p,
		\end{equation}
		where the constant $C_p$ depends only on $p$ and $\lambda_1(M)$. 
		\begin{proof}
			Assume that $p \ge 2$.   Multiplying the equation by $u^{p-1}$ and integrating it, we have
			\begin{align*}
				(p-1)\int_{\Omega} |\nabla u|^2u^{p-2}&=(\nabla u, \nabla u^{p-1})=(\nabla f, \nabla u^{p-1})\\
				&\leq (p-1)\int_{\Omega} |\nabla f||\nabla u|u^{p-2}\\
				&\leq (p-1) \left(\int_{\Omega}|\nabla u|^2 u^{p-2} \right)^{1/2}\left(\int_{\Omega}|\nabla f|^2 u^{p-2} \right)^{1/2}.
			\end{align*}
			Thus we have
			\begin{equation*}
				\frac{4}{p^2}\int_{\Omega} |\nabla u^{p/2}|^2 \leq \int_{\Omega} |\nabla f|^2 u^{p-2}\leq \left(\int_{\Omega} |u|^p \right)^{\frac{p-2}{p}}\left(\int_{\Omega} |\nabla f|^p \right)^{\frac{2}{p}}.
			\end{equation*}
			From \eqref{eq:1.2}, we obtain 
			\begin{equation*}
				\left(\frac{4\lambda_1}{p^2}\right)^{\frac{p}{2}}\int_{\Omega} |u|^p \leq \int_{\Omega} |\nabla f|^p.
			\end{equation*}
			The constant $C_p$ depends only on $p$ and $\lambda_1$. The general case can be proved similarly through multiplication by $(\textrm{sgn}\ u)|u|^{p-1}$ and integration. 
		\end{proof}
		
	\end{lemma}
	
	\begin{proof}[Proof of Theorem~\ref{thm:lower-bounded-c-metric}]
		
		From Lemma~\ref{Lem:Mckean}, $M$ has the positive spectrum.	It is a standard result that if the manifold has positive spectrum then there exists a positive symmetric Green's function $G$ on $M$. Moreover, we can always take $G(x,y)$ to be the minimal Green's function constructed using exhaustion of compact subdomains. Hence
		\begin{equation*}
			G(x,y)=\lim_{i \rightarrow \infty } G_i(x,y)>0,
		\end{equation*}
		where $G_i$ is the Dirichlet Green's function of a compact exhaustion $\left\{\Omega_i \right\}_i$ of $M$, and the limit is uniform on compact subsets of $M$. 
		
		Take any (bounded) holomorphic function $f : M \rightarrow \mathbb{D}$. For any relatively compact subdomain $\Omega \subset M$ with the smooth boundary $b\Omega$, we use $f^2$ in Lemma~\ref{Lem:Cheng} and solving the Dirichlet boundary problem with the inequality 
		\begin{equation}\label{eq:caratheodory}
			\left(g(\nabla f^2,\nabla f^2)(x)\right)^{\frac{p}{2}} = \left(4|f(x)|^2 df(\nabla f)(x)\right)^{\frac{p}{2}}	\leq 2^p |f|^p(x) \gamma_M(x;\nabla f(x))^{\frac{p}{2}}
		\end{equation}
		for any $x\in M$,  and the condition $p\geq 2$ implies
		\begin{equation}\label{eq:estimate1}
			\int_{\Omega} |u|^p \leq \left(\frac{2p}{(2n-1)a} \right)^{p} \int_{\Omega} |f|^p \gamma_M(.;\nabla f)^{\frac{p}{2}} \leq \left(\frac{2p}{(2n-1)a} \right)^{p} \int_{M} |f|^p \gamma_M(.;\nabla f)^{\frac{p}{2}},
		\end{equation}
		where $u$ is the solution of 
		\begin{align}\label{eq:bounded-geometry}
			\begin{cases} \triangle u = 2|\nabla f|^2  & \text{on  } \Omega, \\
				u = 0 & \text{on  } b\Omega,  \end{cases}
		\end{align}
		and $a>0$ is for the upper bound of the Riemannian sectional curvature $\leq -a^2<0$. 
		
		From the hypothesis $|f|^p \gamma_M(.;\nabla f)^{\frac{p}{2}}\in L^1(M)$ and from the exhaustion of compact subdomains, there exists $u \in C^{\infty}(M,\mathbb{R})$ such that 
		\begin{equation*}
			\int_{M} |u|^p < \infty,
		\end{equation*}
		and $\triangle u = 2|\nabla f|^2$ on $M$. Furthermore, the fact $\inf_{x\in M} \vol \, B(x,r)>0$ for any $r>0$  implies that $u(x)\rightarrow 0$ as $d(p,x)\rightarrow \infty$ from some fixed point $p\in M$. Thus the Dirichlet problem is solvable and $u$ can be represented by 
		\begin{equation}\label{eq:sol_of_DP}
			u(x)=2\int_{M} G(x,y)|\nabla f|^2(y)dy,
		\end{equation}
		which proves part (1). 
		
		For part (2), the positive minimal Green's function satisfies
		\begin{equation*}
			G(x,y)=\int_{0}^{\infty}h_M(x,y,t)dt,
		\end{equation*}
		where we denote the heat kernel of the Laplace--Beltrami operator by $h_M(x,y,t)$. Hence ~\eqref{eq:sol_of_DP} becomes
		\begin{equation}\label{eq:ing0}
			u(x)=2\int_{0}^{\infty}\int_{M} h_M(x,y,t)|\nabla f|^2(y)dydt.
		\end{equation}
		
		We use the Cheeger and Yau's heat kernel comparison theorem \cite{CheegerJeffYauShingTung81}:
		\begin{equation}\label{eq:ing1}
			h_M(x,y,t)\geq h_{M_k}(d(x,y)),
		\end{equation}
		where $M_k$ is the space form with constant sectional curvature equal to $k$. From the two-sided estimate of Davies and Mandouvalos \cite{DaviesEBMandouvalosN1988},
		\begin{equation}\label{eq:ing2}
			c(n)^{-1}h(t,d(x,y))\leq h_{M_k}(d(x,y)) \leq c(n)h(t,d(x,y)),
		\end{equation}
		where $c(n)$ depends only on $n$ and 
		\begin{align}\label{eq:ing3}
			h(t,r)&=(2\pi t)^{-n}\exp \left [-\frac{r^2}{2t}-\frac{(2n-1)^2b^2t}{8}-\frac{(2n-1){b}r}{2} \right ] (1+{b}r)\left(1+{b}r+\frac{b^2t}{2} \right)^{\frac{2n-1}{2}-1}
		\end{align} for $t,r>0$,
		where $b>0$ is for the lower bound of the Riemannian sectional curvature $\geq -b^2$. 
		
		Now combining ~\eqref{eq:estimate1} with ~\eqref{eq:ing0},~\eqref{eq:ing1},~\eqref{eq:ing2}, and ~\eqref{eq:ing3} gives the desired inequality ~\eqref{eq:gradient-estimate}. This completes the proof.
	\end{proof}
	
	We end this paper with an example for Theorem ~\ref{thm:lower-bounded-c-metric}.
	\begin{prop}\label{prop:disk-case}
		In the case of unit disk $\mathbb{D}$ in $\mathbb{C}$, for each $p\geq 2$, we have  
		\begin{equation*}
			2\pi \int_{0}^{1}\left(\frac{1}{6}-\frac{R^2}{2}\ln R -\frac{R^4}{8}(4\ln R -1)-\frac{R^6}{36}(6\ln R-1)   \right)^{p} R\, dR\leq p^p \int_{\mathbb{D} }|z|^p \gamma_{\mathbb{D}}(z;\nabla z)^{\frac{p}{2}}.
		\end{equation*}
	\end{prop}
	
	\begin{proof}
		The Green function of  the unit disk $\mathbb{D}$ in $\mathbb{C}$ has the following form:
		\begin{equation*}
			G(x,y)=\frac{1}{2\pi}\ln \frac{|x-y|}{|x||y-\frac{x}{|x|^2}|}. 
		\end{equation*}
		The function $G$ satisfies $\triangle_x G(x,y)=\delta_y$ at fixed $y\in \mathbb{D}$ and $G(x,y)=0$ when $|x|=1$ and $|y|<1$. Since the gradient vector of $z \in \mathbb D$ with respect to the Poincar\'e metric is $(1-|z|^2)\frac{\partial}{\partial z}$, the integrand of the left-hand side of ~\eqref{eq:gradient-estimate2} is 
		\begin{equation}\label{eq:integrand}
			\int_{|y|< 1} G(x,y) (1-|y|^2)^2  dy. 
		\end{equation}
		Rewrite $G(x,y)=\frac{1}{4\pi}\ln \left(\frac{|x|^2|y-x/|x|^2|^2 }{|x-y|^2} \right)$ and choose  coordinates $x=(R,0)$ and $y=(r\cos \theta, r\sin \theta)$, then ~\eqref{eq:integrand} becomes 
		\begin{align*}
			&\frac{1}{4\pi}\int_{0}^{1}\int_{0}^{2\pi}\ln\left(\frac{1+r^2R^2-2rR\cos \theta}{R^2+r^2-2rR\cos \theta} \right)r(1-r^2)^2 d\theta dr\\
			&=\frac{1}{4\pi}\int_{0}^{1}r(1-r^2)^2 \left( I(1,rR)-I(r,R) \right)dr,
		\end{align*}
		where $I(a,b):=\int_{0}^{2\pi}\ln(a^2+b^2-2ab\cos \theta)d\theta$. It is well-known that 
		\begin{equation*}
			I(a,b)=4\pi \max \left\{\ln |a|, \ln |b| \right\}. 
		\end{equation*}
		Since $0\leq r,R\leq 1$, we have $I(1,rR)=0$. Thus the integral becomes
		\begin{align*}
			&-\int_{0}^{1}r(1-r^2)^2 \max \left\{\ln |r|, \ln |R| \right\} dr =  -\ln R \int_{0}^{R}r(1-r^2)^2dr-\int_{R}^{1}r(1-r^2)^2 \ln r dr\\
			&=\frac{1}{6}-\frac{R^2}{2}\ln R -\frac{R^4}{8}(4\ln R -1)-\frac{R^6}{36}(6\ln R-1). 
		\end{align*}
		Thus the left-hand side of ~\eqref{eq:gradient-estimate2} is 
		\begin{equation*}
			2\pi \int_{0}^{1}\left(\frac{1}{6}-\frac{R^2}{2}\ln R -\frac{R^4}{8}(4\ln R -1)-\frac{R^6}{36}(6\ln R-1)   \right)^{p} R\, dR .
		\end{equation*}
	\end{proof}	
	
	\subsection*{Conflicts of interest}
	The corresponding author states that there is no conflict of interest. 
	
	\bibliographystyle{spmpsci}
	\bibliography{reference}

	%\begin{thebibliography}{99}
	%\bibliographystyle{spmpsci}
	%\bibliography{reference}
	
	%\end{thebibliography}
	
\end{document}